\documentclass[11pt]{amsart}  

\usepackage{amsmath}
\usepackage{amsthm}
\usepackage{amsfonts}
\usepackage{amssymb}
\usepackage{eucal}
\usepackage[all]{xy}
\usepackage{amsxtra}  
\usepackage[pdftex, breaklinks, linktocpage=true, bookmarksopen=true,bookmarksopenlevel=0,bookmarksnumbered=true]{hyperref}

\hypersetup{
pdftitle = {Canonical extensions of N\'eron models of Jacobians},
pdfsubject = {Canonical extensions of N\'eron models of Jacobians},
pdfkeywords = {Integral structure, de~Rham cohomology, abelian variety, Jacobian, N\'eron model, canonical extension, rigidified extensions, group scheme, picard functor}
pdfauthor = {\textcopyright\ Bryden Cais},
pdfcreator = {\LaTeX\ with package \flqq hyperref\frqq},
colorlinks = {true}
}

\xymatrixcolsep{2.8pc}
\CompileMatrices

\DeclareFontFamily{OT1}{rsfs}{}
\DeclareFontShape{OT1}{rsfs}{n}{it}{<-> rsfs10}{}
\DeclareMathAlphabet{\mathscr}{OT1}{rsfs}{n}{it}

\makeatletter
\newcommand{\rmnum}[1]{\romannumeral #1}
\newcommand{\Rmnum}[1]{\expandafter\@slowromancap\romannumeral #1@}
\makeatother

\DeclareMathOperator{\Ann}{Ann}
\DeclareMathOperator{\depth}{depth}

\DeclareMathOperator{\id}{id}

\DeclareMathOperator{\Frac}{Frac}
 
\DeclareMathOperator{\Sym}{Sym} 
 
\DeclareMathOperator{\pr}{pr}
 
\DeclareMathOperator{\Hom}{Hom}
\DeclareMathOperator{\End}{End}
\DeclareMathOperator{\Ext}{Ext}

\DeclareMathOperator{\Spec}{Spec}

\DeclareMathOperator{\dR}{dR}

\DeclareMathOperator{\Pic}{Pic}
\DeclareMathOperator{\Alb}{Alb}
\DeclareMathOperator{\sm}{sm}
\DeclareMathOperator{\Extrig}{Extrig}
\DeclareMathOperator{\Lie}{Lie}
\DeclareMathOperator{\Inf}{Inf}

\DeclareMathOperator{\Char}{char}

\DeclareMathOperator{\sh}{sh}

\newcommand*{\R}{\ensuremath{\mathbf{R}}}              
\newcommand*{\Z}{\ensuremath{\mathbf{Z}}}               
\newcommand*{\Q}{\ensuremath{\mathbf{Q}}}                           

\newcommand*{\Kbar}{\overline{K}}    
          
\newcommand*{\Gm}{\ensuremath{{\mathbf{G}_m}}}   
\newcommand*{\Ga}{\ensuremath{{\mathbf{G}_a}}}   

\newcommand*{\m}{\mathfrak{m}}

\newcommand*{\E}{\mathscr{E}}     
\newcommand*{\F}{\mathscr{F}}
\newcommand*{\G}{\mathscr{G}}

\renewcommand*{\L}{\mathscr{L}}
\newcommand*{\M}{\mathscr{M}}
  
\renewcommand*{\O}{\mathscr{O}}

\newcommand*{\scrHom}{\mathscr{H}\mathit{om}}      
\newcommand*{\scrExtrig}{\mathscr{E}\mathit{xtrig}}	
\newcommand*{\scrExt}{\mathscr{E}\mathit{xt}}               
\newcommand*{\scrLie}{\mathscr{L}\mathit{ie}}                

\renewcommand*{\H}{\ensuremath{\mathbf{H}}}
 
\newcommand*{\dual}{\vee}

\newcommand*{\Dual}[1]{\widehat{#1}}

\renewcommand*{\int}{\ensuremath{\mathrm{int}}}

\theoremstyle{plain}
  \newtheorem{theorem}{Theorem}
  \newtheorem{proposition}[theorem]{Proposition}
  \newtheorem{lemma}[theorem]{Lemma}
  \newtheorem{corollary}[theorem]{Corollary}

\theoremstyle{definition}
  \newtheorem{definition}[theorem]{Definition}

  \newtheorem{question}[theorem]{Question}

\theoremstyle{remark}
  
  \newtheorem{remark}[theorem]{Remark}
  \newtheorem{remarks}[theorem]{Remarks}

\include{header}

\numberwithin{theorem}{section}  
\numberwithin{equation}{section}

\begin{document}
\bibliographystyle{amsplain_noMR}

\title[{Canonical extensions of N\'eron models of Jacobians}]{Canonical extensions of N\'eron models of Jacobians}

\author{Bryden Cais}
\address{Centre de recherches math\'ematiques, Montr\'eal}
\curraddr{Dept. of Mathematics and Statistics,  805 Sherbrooke St. West, Montr\'eal, QC, Canada, H3A 2K6}
\email{bcais@math.mcgill.ca}
\thanks{This work was partially supported by the NSF grant DMS-0502170 and by a Rackham Predoctoral Fellowship.}

\date{\today}

\begin{abstract}
	Let $A$ be the N\'eron model of an abelian variety $A_K$ over the fraction field 
	$K$ of a discrete valuation ring $R$.  Due to work of Mazur-Messing, there is a functorial
	way to prolong the universal extension of $A_K$ by a vector group to a smooth and separated 
	group scheme over $R$, called the {\em canonical extension of $A$}.  
	In this paper, we study the canonical extension when $A_K=J_K$ is the Jacobian
	of a smooth proper and geometrically connected curve $X_K$ over $K$.  Assuming that 
	$X_K$ admits a proper flat regular model $X$ over $R$ that has generically smooth closed fiber,
	our main result identifies the identity component of the canonical extension with a certain functor  
	$\Pic^{\natural,0}_{X/R}$ classifying line bundles on $X$ that have partial degree zero on all components of geometric fibers and are equipped with
	a {\em regular connection}.  This result is a natural extension of a theorem of Raynaud, which
	identifies the identity component of the N\'eron model  $J$ of $J_K$ with the functor $\Pic^0_{X/R}$.
	As an application of our result,
	we prove a comparison isomorphism between two canonical integral structures on the de~Rham cohomology
	of $X_K$.
\end{abstract}

\maketitle
\section{Introduction}\label{intro}

	  Fix a discrete valuation ring $R$ with field of fractions $K$ and residue field $k$.
	  Let $A_K$ be an abelian variety over $K$ and consider the universal extension $E(\Dual{A}_K)$ of the dual abelian variety $\Dual{A}_K$.
	 This commutative algebraic $K$-group is an extension of $\Dual{A}_K$ by the vector group of invariant differentials on $A_K$
	  \begin{equation}
		\xymatrix{
			0\ar[r] & {\omega_{A_K}} \ar[r] & {E(\Dual{A}_K)} \ar[r] & {\Dual{A}_K} \ar[r] & 0
			}\label{universal}
	\end{equation}   
	 and is universal among extensions of $\Dual{A}_K$ by a vector group:  for any vector group $V$ over $K$, the natural homomorphism
	 $\Hom(\omega_{A_K},V)\rightarrow \Ext(\Dual{A}_K,V)$ arising by pushout from (\ref{universal}) is an isomorphism.	
	  The theory of the universal extension was initiated by Rosenlicht \cite{Rosenlicht}, who defined the notion and showed its existence for abelian varieties,
	  and subsequently taken up by Tate \cite{Tate}, Murre \cite{Murre}, Grothendieck \cite{GrothBarsotti}, Messing \cite{Messing}, \cite{Messing2} and Mazur-Messing \cite{MM}.  
	  It plays a central role in the definition of the Mazur-Tate $p$-adic height pairing \cite{MazTate}, \cite{coleman-biext}, 
	  in Deligne's definition of the duality on the de~Rham cohomology 
	  of $A_K$ \cite[\S 10.2.7.3]{DeligneHodge3} (see also the articles of Coleman \cite{coleman-biext}, \cite{colemanduality}), 
	  and in certain proofs of the comparison isomorphism between the  $p$-adic \'etale and de~Rham cohomologies
	  of $A_K$ \cite{FMUnpub}, \cite[Note added in proof]{ColemanComparison}, \cite{Wintenberger}.
	    	
	  As is well-known, the N\'eron model $\Dual{A}$ of $\Dual{A}_K$ over $R$ provides a functorial extension of $\Dual{A}_K$ to a smooth commutative
	  group scheme over $R$, and it is natural to ask if (\ref{universal}) can be functorially extended 
	  to a short exact sequence of smooth commutative $R$-groups as well. 
	  Such an extension is provided by the ``canonical extension" $\E(\Dual{A})$ of $\Dual{A}$, introduced by Mazur and Messing in \cite[I,\S5]{MM}.
	     When $\Dual{A}_K$ has good reduction, $\E(\Dual{A})$ coincides with the universal extension of (the abelian scheme) $\Dual{A}$ by a vector 
	  group, but in general, as an example of Breen and Raynaud shows ({\em cf.} Remarks \ref{unicaneq}), N\'eron models
	  need not have universal extensions, and $\E(\Dual{A})$ seems to be the best substitute in such cases.  
	    Although they seem to be of fundamental importance, canonical extensions of N\'eron models
	    have been little studied, and as far as we know, do not appear anywhere in the literature beyond their 	
	    introduction in \cite{MM} and in \cite[\S15]{tameness}.
	     
	    In this paper, we study the canonical extension $\E(\Dual{A})$ when $A_K=J_K$ is the 
	    the Jacobian of a smooth proper and geometrically connected curve $X_K$ over $K$.
	      In this situation, a famous theorem of Raynaud (see \cite[\S 9.7, Theorem 1]{BLR}) relates the identity component 
	      $\Dual{J}^0$ of $\Dual{J}$ to the relative Picard functor of any proper flat and normal model $X$ of $X_K$ that is ``sufficiently nice":
	           \begin{theorem}[Raynaud]\label{Rthm}
	           	Let $S=\Spec R$ and fix a proper flat and normal model $X$ of $X_K$ over $S$.  Denote by $X_1,\ldots,X_n$
				the $($reduced$)$ irreducible components of the closed fiber $X_k$.  
				Suppose that the greatest common divisor of the geometric multiplicities of the $X_i$ in $X_k$ is equal to $1$,
				and assume either that $k$ is perfect or that $X$ admits an \'etale quasi-section.
			Then $\Pic^0_{X/S}$ is a smooth and separated $S$-group scheme and $J_K$ admits a N\'eron model $J$ of finite type.
				Moreover, the canonical morphism
			\begin{equation}
				\xymatrix{
					{\Pic^0_{X/S}} \ar[r] & {\Dual{J}^0}
				}\label{Rthmmap}
			\end{equation}
				arising via the N\'eron mapping property from the canonical principal polarization of $J_K$
			is an isomorphism if and only if $X$ has rational singularities.\footnote{
				Recall that $X$ is said to have {\em rational singularities} if it admits a resolution of singularities
				 $\rho:X'\rightarrow X$ with $R^1\rho_*\O_{X'}=0$.  Trivially, any regular $X$ has rational
				 singularities.}   
	  \end{theorem}
	  		  	
	   Our main result enhances Raynaud's theorem by providing a similar description of the identity component $\E(\Dual{J})^0$ 
	    of the canonical extension $\E(\Dual{J})$ of $\Dual{J}$:
	     
	  \begin{theorem}\label{main}
	  		Let $X$ be a proper flat and normal model of $X_K$ over $S=\Spec R$.  Suppose that the closed fiber of $X$ is
			geometrically reduced and that either $X$ is regular or that $k$ is perfect.		
			Then there is a canonical homomorphism of short exact sequences of smooth group schemes over $S$
		    \begin{equation}
		    	\xymatrix{		
				0\ar[r]  & \omega_{J} \ar[r]\ar[d] & {\E(\Dual{J})^0} \ar[r]\ar[d] &  {\Dual{J}^0} \ar[r]\ar[d] & 0 \\
					0\ar[r]  & {f_*\omega_{X/S}} \ar[r] & {\Pic^{\natural,0}_{X/S}} \ar[r] & {\Pic^0_{X/S}}\ar[r] & 0
				}\label{maindiag}
		\end{equation}
		which is an isomorphism of exact sequences if and only if $X$ has rational singularities.
	  \end{theorem}

	Here,  $\omega_{X/S}$ is the relative dualizing sheaf of $X$ over $S$; it is a coherent sheaf of $\O_X$-modules
	that is flat over $S$ and coincides with the sheaf of relative differentials over the smooth locus of $f$ in $X$. 
	We write $f_*\omega_{X/S}$ for the vector group attached to this locally free $\O_S$-module,
	and $\Pic^{\natural,0}_{X/S}$ is the fppf sheaf associated to the functor on $S$-schemes
	that assigns to each $S$-scheme $\varphi:T\rightarrow S$ the set of isomorphism classes of pairs $(\L,\nabla)$, where $\L$ is a line bundle on $X_T$
	whose restriction to all components of each geometric fiber of $X_T$ has degree zero and 
	$\nabla:\L\rightarrow \L\otimes \varphi^*\omega_{X/S}$ is a {\em regular connection} on $\L$ over $T$ (Definition \ref{regconn}).    
	We will show in Theorem \ref{rep} that under the hypotheses of Raynaud's Theorem, $\Pic_{X/S}^{\natural,0}$ is indeed a smooth and separated $S$-scheme, and that
	there is a short exact sequence of smooth groups over $S$ as in the lower row of (\ref{maindiag}).

	    We note that when $f:X\rightarrow S$ is smooth, our notion of regular connection coincides with the familiar notion of
	    connection, and we recover from Theorem \ref{main} the ``well-known" description of the universal extension of a Jacobian
	    of a smooth and proper  curve as the representing object of the functor classifying degree zero line bundles on the curve that are
	    equipped with connection.\footnote{Certainly this result appears in the literature 
	    (see, for example \cite[\S 2]{coleman-ext}), however, we have been unable to find any proof of it.  See, however \cite[\Rmnum{1} \S4]{MM},
	    which proves a result in a similar spirit.}
	    Let us also point out that the hypotheses of Theorem \ref{main} include not only all regular curves over $K$ with semistable reduction but many regular curves which are
	    quite far\footnote{In the sense that they achieve semistable reduction only after a large and wildly ramified extension of $K$.} from having semistable reduction,
	    such as the modular curves $X(N)$ and $X_1(N)$ over $K:=\Q_p(\zeta_{N})$ for {\em arbitrary} $N$ (see Theorems 13.7.6 and 13.11.4 of \cite{KM},
	     which describe proper flat and regular models of $X(N)$ and $X_1(N)$, respectively, over $R=\Z_p[\zeta_N]$ that have generically smooth---even
	     geometrically reduced---closed fibers).

	      It is well-known that the exact sequence of Lie algebras arising from (\ref{universal})
	    is naturally isomorphic to the 3-term Hodge filtration exact sequence of the first de~Rham cohomology of $A_K$ (Proposition \ref{LieUniv}).  Thus, 
	     the Lie algebra of the smooth $R$-group $\E(\Dual{A})$ 
	      provides a canonical $R$-lattice in the $K$-vector space $H^1_{\dR}(A_K/K)$ which is functorial in $K$-morphisms
	    of $A_K$ (due to the N\'eron mapping property of $A$ and the functorial dependence of $\E(\Dual{A})$ on $A$).  When $A$ is an abelian scheme
	      and the maximal ideal of $R$ has divided powers, Mazur-Messing proved 
		 \cite[\Rmnum{2} \S15]{MM} Grothendieck's conjecture \cite[\Rmnum{5} \S5]{GrothBarsotti} that this $R$-lattice is naturally isomorphic to the Dieudonn\'e module of 
	   the associated Barsotti-Tate group $A_k[p^{\infty}]$.  Thus, $\Lie(\E(A))$ provides a natural generalization of the Dieudonn\'e module when $A$
	   is not an abelian scheme.
	   In \cite{CaisDualizing}, for a proper flat and normal $R$-curve $X$, we studied a canonical integral structure $H^1(X/R)$ on $H^1_{\dR}(X_K/K)$
	   (i.e. an $R$-lattice that is functorial in $K$-morphisms of $X_K$) defined in terms of relative dualizing sheaves.
	 	  It is natural to ask how $H^1(X/R)$ compares with the lattice $\Lie( \E(\Dual{J}))$ under the canonical identification
	   $H^1_{\dR}(X_K/K) \simeq H^1_{\dR}(J_K/K)$.  We will prove in Corollary \ref{integralcompare} 
	   that these two lattices coincide  when $X$ verifies the hypotheses of Theorem \ref{main}.

	 We briefly explain the main ideas that underlie the proof of Theorem \ref{main}.  Our first task is 
	 to reinterpret $\E(\Dual{J})^0$ as
	 the representing object of the functor $\scrExtrig_S(J,\Gm)$ on smooth $S$-schemes, {\em \`a la} 
	 Mazur-Messing \cite{MM}.  To do this, we must first show that the functor $\scrExt_S(J,\Gm)$
	 is represented by $\Dual{J}^0$ on smooth $S$-schemes, and by \cite[Proposition 5.1]{BoschComponent} this holds 
	 if and only if Grothendieck's pairing on component groups is perfect.   
	 It follows from results of Bosch and Lorenzini  \cite[Corollary 4.7]{BoschLorenzini} 
	 (see also Proposition \ref{GrPerfJ}) that the hypotheses of Theorem \ref{main} imply
	 the perfectness of Grothendieck's pairing.
	 However, we note that Grothendieck's pairing is {\em not}
	 generally perfect (see Remark \ref{GrPerf}).
	 
	 In \S\ref{dualizingsec}, we construct the exact sequence
	 of smooth $S$-group schemes occurring in the bottom row of (\ref{maindiag}).
	 This is accomplished by Theorem \ref{rep}, whose proof employs \v{C}ech-theoretical techniques to
	 interpret the hypercohomology of the two-term complex $d\log: \O_{X}^{\times}\rightarrow \omega_{X/S}$ in terms 
	 of line bundles with regular connection,
	 and makes essential use of the good cohomological properties of the relative dualizing sheaf and 
	 of Grothendieck duality.
	 A key insight here is that the traditional notion of a connection on a line bundle on a scheme 
	 $X$ over a base $S$ is not well-behaved when $X$ is not $S$-smooth and must be suitably modified as in 
	 Definition \ref{regconn}.
	 With these preliminaries in place, we turn to the proof of Theorem \ref{main} in \S\ref{mainthmpf}.  
	 We must first construct a morphism of
	 short exact sequences of smooth group schemes (\ref{maindiag}).  
	 Our strategy for doing this is as follows.  Passing to an unramified extension
	 of $K$ if need be, we suppose
	 that $X_K$ has a rational point and use it to define an albanese morphism $j_K:X_K\rightarrow J_K$.  
	 The N\'eron mapping property of
	 $J$ allows us to extend $j_K$ to a morphism $j:X^{\sm}\rightarrow J$ on the smooth locus of $f$ in $X$.  
	 By (functorially) pulling back rigidified extensions of $J$ by $\Gm$ along $j$, we get line bundles on $X^{\sm}$
	 with connection.  
	 Via a careful analysis
	 of the relative dualizing sheaf, we show in 
	 Lemma \ref{reduce2U} that a line bundle with connection
	 on $X^{\sm}$ is equivalent to a line bundle with regular connection on $X$; this critically uses
	 our hypothesis that the closed fiber of $X$ is geometrically reduced (equivalently, that $X^{\sm}$
	 is fiber-wise dense in $X$).  From this, we deduce the desired map (\ref{maindiag}).
	 To complete the proof of Theorem \ref{main}, we then ``bootstrap" Raynaud's theorem \ref{Rthm} using duality.  
	 Here, it is essential to know that the canonical evaluation duality between the Lie algebra of $J$ and the 
	 sheaf of invariant differentials
	 on $J$ is compatible via $j$ with the (Grothendieck) duality of $f_*\omega_{X/S}$ and $R^1f_*\O_X$.  
	 Such compatibility may be checked
	 on generic fibers, where it is well-known (e.g. \cite[Theorem 5.1]{colemanduality}). 
	 
	 We remark that when $k$ is perfect, both the short exact sequences of group schemes in the rows of (\ref{maindiag})
	 exist under the less restrictive hypotheses of Theorem \ref{Rthm}; 
	this follows immediately from Propositions \ref{repExtrig} and \ref{GrPerfJ} for the top
	row of (\ref{maindiag}), and from Theorem \ref{rep} for the bottom row.  
	 It is natural to ask if Theorem \ref{main} holds in this generality as well.
	 We do not know the answer to this question, as our construction of the map of
	 short exact sequences of smooth groups in (\ref{maindiag}) seems to require
	 the closed fiber of $X$ to be generically smooth.
	 Indeed, our construction of (\ref{maindiag}) relies on extending an albanese morphism 
	 $X_K\rightarrow J_K$ to some 
	 open subscheme $U$ of $X$ with the property that line bundles with connection on $U$ uniquely 
	 extend to line bundles with regular connection on $X$.  On the one hand, this extension property
	 seems to require $U$ to be fiber-wise dense in $X$ (see Lemma \ref{reduce2U} and Remark \ref{reduce2Ufalse}), while
	 on the other hand one only expects to be able to extend the morphism $X_K\rightarrow J_K$ to 
	 $U=X^{\sm}$.  Thus, we are forced to require that $U=X^{\sm}$ be fiber-wise dense in $X$, {\em i.e.} that
	 $X_k$ be generically smooth (equivalently geometrically reduced).  We note, however, that it is just
	 our construction of the map (\ref{maindiag}) that requires $X$ to have generically smooth closed fiber; 
	 the proof that this map is an isomorphism of exact sequences of group schemes relies only on the weaker
	 hypotheses of Raynaud's Theorem \ref{Rthm}.
	 
	 It is a pleasure to thank Dino Lorenzini, Barry Mazur, and Bill Messing for several 
	 helpful exchanges and conversations.  
	 Many thanks to C\'edric P\'epin, both for pointing out the argument of Proposition \ref{GrPerfJ} to me,
	 and for providing me with a copy of \cite{Pepin}.
	I would especially like to thank Brian Conrad
	for advising my Ph.D. thesis (which was the genesis of this article) and for many clarifications.
	I am very grateful to the referee for making several corrections and for suggesting
	a number of improvements.

\bigskip
\noindent {\bf Conventions and Notation:}

Fix a base scheme $S$.  If $Y$ is any $S$-scheme and $S'\rightarrow S$ is any morphism, we will
often write $Y_{S'}:=Y\times_S S'$ for the base change.  When $S'=\Spec(F)$ is the spectrum of a field,
we will sometimes abuse notation and write $Y_F$ in place of $Y_{S'}$.
We will work with the fppf topology on the categories of $S$-schemes and of smooth $S$-schemes (see
\cite[Exp. IV, \S6.3]{SGA3vol1} or \cite[\S8.1]{BLR}); if $\F$ is any representable functor on one of these categories, we will
also write $\F$ for the representing object.
By an {\em $S$-group scheme $G$} we will always mean a finitely presented flat and separated commutative group scheme over $S$.
As usual, we write $\Ga$ and $\Gm$ for the additive and multiplicative group schemes over $S$.
A {\em vector group} on $S$ is any $S$-group that is Zariski-locally isomorphic to a product of $\Ga$'s. 
Associated to any quasi-coherent $\O_S$-module $\M$ is a sheaf for the fppf topology on $S$-schemes $\varphi:T\rightarrow S$ given by 
$\M(T):=\Gamma(T,\varphi^*\M)$.  When $\M$ is locally free of finite rank,
this fppf sheaf is represented by the vector group $\mathbf{Spec}\left( \Sym_{\O_S}(\M^*)\right)$ where $\M^*$ is the $\O_S$-linear dual of $\M$; 
we will frequently abuse notation and write $\M$ for both the locally free $\O_S$-module and the associated vector group on $S$.  
For any $S$-group $G$ with identity section $e:S\rightarrow G$, we put $\omega_G:=e^*\Omega^1_{G/S}$. 
As usual, for any $S$-scheme $T$ we put $T[\epsilon]:= T\times_{\Z} \Spec(\Z[\epsilon]/\epsilon^2)$, considered as a $T$-scheme via the first projection,
and	 for any fppf sheaf $G$ we write $\scrLie(G)$ for the fppf sheaf of $\O_S$-modules defined (as in \cite[\S1]{LiuNeron}) by
$\scrLie(G)(T) := \ker(G(T[\epsilon])\rightarrow G(T))$.
When $G$ is a smooth group, this agrees with the traditional notion of relative Lie algebra (as a sheaf of $\O_S$-modules).  We
set $\Lie(G):=\scrLie(G)(S)$.

\tableofcontents

\section{Canonical Extensions of N\'eron models}\label{canext}

In this section, following \cite{MM}, we recall the construction and basic properties of the the canonical extension of a N\'eron model,
and we explain how to interpret its identity component via rigidified extensions.

Let $S$ be any base scheme, and fix commutative $S$-group schemes $F$ and $G$.  A {\em rigidified extension}
of $F$ by $G$ over $S$ is a pair $(E,\sigma)$ consisting of an extension $E$ (of fppf sheaves of abelian groups over $S$) of $F$ by $G$
\begin{equation}
	\xymatrix{
		0 \ar[r] & G \ar[r]^-{\iota} & E \ar[r] & F\ar[r] & 0
	}\label{ext}
\end{equation}
and a section $\sigma$ of $S$-pointed $S$-schemes along the first infinitesimal neighborhood of the identity of $F$
\begin{equation}
	\xymatrix{
		{\Inf_S^1(F)} \ar[r]^-{\sigma} & E
		}\label{rig}
\end{equation}
that projects to the canonical closed immersion $\Inf_S^1(F)\rightarrow F$.
Two rigidified extensions $(E,\sigma)$ and $(E',\sigma')$ of $F$ by $G$ are called {\em equivalent} if
there is a homomorphism (necessarily an isomorphism) $\varphi: E\rightarrow E'$ that carries $\sigma$ to $\sigma'$ and makes the diagram
\begin{equation}
	\xymatrix{
		0\ar[r] & G \ar[r]^-{\iota}\ar@{=}[d] & E \ar[r]\ar[d]^-{\varphi} & F \ar[r]\ar@{=}[d] & 0\\
		0\ar[r] & G \ar[r]_-{\iota'} & E' \ar[r] & F \ar[r] & 0
	}\label{extequiv}
\end{equation}
commute.

We denote by $\Extrig_S(F,G)$ the set of equivalence classes of rigidified extensions of $F$ by $G$ over $S$.
This set is equipped with a natural group structure via Baer sum of rigidified extensions (see \cite[\Rmnum{1} \S2.1]{MM}) which makes the
functor on $S$-schemes $T\rightsquigarrow \Extrig_T(F_T,G_T)$ a group functor that is contravariant in the first variable via pullback
(fibered product) and covariant in the second variable via pushout (fibered co-product).  We will write $\scrExtrig_S(F,G)$
for the fppf sheaf of abelian groups associated to this functor.

We will exclusively be concerned with the special case that $G=\Gm$ is the multiplicative group over $S$.  
Note that (by fppf descent) any extension of $F$ by $\Gm$ is automatically representable as $\Gm$ is affine ({\em cf.}
the proof of \cite[\Rmnum{3}, Proposition 17.4]{oort}).
In this context, there is an alternate and more concrete functorial description of the group $\Extrig_S(F,\Gm)$ that we will need for later use.
Fix a choice of generator $\tau$ for the free rank-one $\Z$-module of invariant differentials $\omega_{\Gm}$ of $\Gm$ over $\Z$.
Note that $\tau$ is canonically determined up to multiplication by $\pm 1$.  For any scheme $S$, we will denote the pullback
of $\tau$ to a generator of $\omega_{\Gm}$ simply by $\tau$.  Write $\mathbf{E}_{\tau}(F)(S)$ for the set of equivalence classes of pairs $(E,\eta)$
consisting of an extension $E$ of $F$ by $\Gm$ over $S$ and a global invariant differential $\eta\in \Gamma(S,\omega_{E})$
which pulls back via the given morphism $\iota:\Gm\rightarrow E$ (realizing $E$ as an extension of $F$ by $\Gm$)
to $\tau$ on $\Gm$.  Two pairs $(E,\eta)$ and $(E',\eta')$ are declared to be equivalent if there is a morphism
$\varphi:E\rightarrow E'$ inducing a diagram as in (\ref{extequiv}) and having the property that $\varphi^*\eta'=\eta$.  
We make $\mathbf{E}_{\tau}(F)(S)$ into an abelian group as follows.  Let $(E,\eta)$ and $(E',\eta')$ be two pairs as above, and denote by $E''$ the Baer sum of
$E$ and $E'$.  Writing $\pr,\pr'$ for the projections from $E\times_F E'$ to $E$ and $E'$, and denoting by $q:E\times_FE'\rightarrow E''$ the quotient map,
we claim that there is a unique invariant differential $\eta''$ on $E''$ satisfying 
$$q^*\eta'' = \pr^*\eta + {\pr'}^*\eta'.$$
Indeed, by definition, $E''$ is the cokernel of the skew-diagonal $(\iota,-\iota'):\Gm\rightarrow E\times_F E'$ under which $\pr^*\eta + {\pr'}^*\eta'$
pulls back to zero.  Thus, via the short exact sequence 
\begin{equation*}
	\xymatrix{
		0\ar[r] & \omega_{E''} \ar[r] & \omega_{E\times_F E'} \ar[r] & \omega_{\Gm} \ar[r] & 0
	}
\end{equation*}
(which is left exact since $E\times_F E'\rightarrow E''$ is smooth due to \cite[Exp.~$\mathrm{\Rmnum{6}}_{\mathrm{B}}$, Proposition 9.2 \rmnum{7}]{SGA3vol1}) 
we obtain a unique invariant differential $\eta''$ on $E''$ as claimed.  One easily checks
that under the map $\Gm\rightarrow E''$ induced by either one of the inclusions $(\iota,0), (0,\iota'): \Gm \rightrightarrows E\times_F E'$
(whose composites with $q$ both coincide with the inclusion $\Gm\rightarrow E''$ realizing $E''$ as an extension of $F$ by $\Gm$)
the differential $\eta''$ pulls back to $\tau$.  We define the sum of the classes represented by $(E,\eta)$ and $(E',\eta')$ to be the class represented by
$(E'',\eta'')$.  It is straightforward to verify that this definition does not depend on the choice of representatives, and makes $\mathbf{E}_{\tau}(F)(S)$
into an abelian group.  This construction is obviously contravariantly functorial in $S$ via pullback of extensions and of invariant differentials.

\begin{lemma}\label{extrigaltdesc}
	For any choice of basis $\tau$ of $\omega_{\Gm}$, there is a functorial isomorphism of abelian groups
	\begin{equation*}
		\xymatrix{
			{\Extrig_S(F,\Gm)} \ar[r]^-{\simeq} & {\mathbf{E}_{\tau}(F)(S)}
		}.
	\end{equation*}
\end{lemma}

\begin{proof}
	 Associated to the extension (\ref{ext}) with $G=\Gm$ is the short exact sequence of Lie algebras
	\begin{equation}
		\xymatrix{
			0\ar[r] & {\scrLie (\Gm) }\ar[r] & {\scrLie (E)} \ar[r] & {\scrLie (F)} \ar[r]  & 0
		}\label{Liesplit}
	\end{equation}	
    	(note that the map $\scrLie (E)\rightarrow \scrLie (F)$ is surjective by \cite[Proposition 1.1 (c)]{LiuNeron}, as $E\rightarrow F$ is smooth).
	We claim that the data of a rigidification on (\ref{ext}) is equivalent to a choice of a splitting of (\ref{Liesplit}).
	Indeed, any map $\sigma: \Inf^1_S(F)\rightarrow E$ necessarily factors through $\Inf^1_S(E)$, so using the natural
	isomorphism $\Inf^1_S(H)\simeq \mathbf{Spec}(\O_S[\omega_H])$ for any smooth group scheme $H$ over $S$ (with $\omega_H$
	the $\O_S$-module of invariant differentials) we obtain a bijection between rigidifications of (\ref{ext}) and  
	sections $\omega_E\rightarrow \omega_F$ to the pullback map $\omega_F\rightarrow \omega_E$.  By the usual duality
	of the $\O_S$-modules $\scrLie(H)$ and $\omega_H$ \cite[Exp.~\Rmnum{2} \S4.11]{SGA3vol1}, this is equivalent to a section $s$ as claimed.  
	 
	 Using $\tau$ to identify the free rank one $\O_S$-module $\scrLie(\Gm)$ with $\O_S$ and
	 thinking of a splitting of (\ref{Liesplit}) as a map $\scrLie (E)\rightarrow \scrLie (\Gm)$ restricting to the identity on $\scrLie (\Gm)$,
	 we see that any such splitting is by duality equivalent 
	 to a global section $\eta \in \Gamma(S,\omega_{E})$ pulling back to $\tau$ in $\Gamma(S,\omega_{\Gm})$.
	One checks that the equivalence $(E,\sigma)\leftrightarrow (E,\eta)$ induces an isomorphism
	of abelian groups $\Extrig_S(F,\Gm)\rightarrow \E_{\tau}(F)(S)$ that is functorial in $S$, as claimed. 
\end{proof}

The following key result shows that the functor $\scrExtrig$ allows one to realize the universal extension of an abelian scheme:

\begin{proposition}[Mazur-Messing]\label{MMunivext}
	Let $A$ be an abelian scheme over an arbitrary base scheme $S$ and denote by $\Dual{A}$ the dual abelian scheme.
	Then the fppf sheaf $\scrExtrig_S(A,\Gm)$
	is a smooth and separated $S$-group scheme.  It sits in a natural short exact sequence of smooth $S$-group schemes
	\begin{equation}
		\xymatrix{
			0\ar[r] & {\omega_{A}} \ar[r] & {\scrExtrig_S(A,\Gm)} \ar[r] & {\Dual{A}} \ar[r] & 0
		}.\label{universalsch}
	\end{equation}
	 Moreover, $(\ref{universalsch})$ is the universal extension of $\Dual{A}$
	by a vector group.
\end{proposition}

\begin{proof}
	See \cite{MM}, especially \Rmnum{1} \S 2.6 and \Rmnum{1} Proposition 2.6.7.
\end{proof}

We now specialize to the case that $S=\Spec R$ is the spectrum of a discrete valuation ring $R$ with field of fractions $K$.
Fix an abelian variety $A_K$ over $K$ and denote by $A$ the N\'eron model of $A_K$ over $S$ and by $A^0$
the relative identity component of $A$.  Let
$\Dual{A}$ be the N\'eron model of the dual abelian variety $\Dual{A}_K$.
In \cite[\Rmnum{1}, 5.2]{MM}, Mazur and Messing prove the following analogue of Proposition \ref{MMunivext}:

\begin{proposition}\label{canonicalext}
	As a functor on smooth $S$-schemes, the fppf abelian sheaf $\scrExtrig_S(A^0,\Gm)$
	is represented by a smooth and separated $S$-group scheme.  Moreover, there
	is a natural short exact sequence of smooth groups over $S$
	\begin{equation}
		\xymatrix{
			0\ar[r] & {\omega_A} \ar[r] & {\scrExtrig_S(A^0,\Gm)} \ar[r] & {\Dual{A}} \ar[r] & 0
		}.\label{canextexseq}
	\end{equation}
\end{proposition}

\begin{definition}[Mazur-Messing]\label{canextdefn}
	The {\em canonical extension} of $\Dual{A}$ is the smooth and separated $S$-group scheme
	$$\E(\Dual{A}):=\scrExtrig_S(A^0,\Gm).$$
\end{definition}

\begin{remarks}\label{unicaneq}
	When $A$ is an abelian scheme, the canonical extension $\E(\Dual{A})$ coincides with the universal extension of $\Dual{A}$ by a vector group
	by Proposition \ref{MMunivext}.  When $A$ is not an abelian scheme, an example of Breen and Raynaud \cite[I, 5.6]{MM}
	shows that $A$ need not have a universal extension.

	Note, however, that since the functor $\scrExtrig_S(A^0,\Gm)$ commutes with fppf base change, the smooth group
	scheme $\scrExtrig_S(A^0,\Gm)$ representing it on the category of smooth group schemes over $S$ is of formation
	compatible with base change to a smooth $S$-scheme.  In particular, the $K$-fiber of the canonical extension 
	exact sequence (\ref{canextexseq}) is the universal extension of $\Dual{A}_K$ by a vector group, thanks to Proposition \ref{MMunivext}.
\end{remarks}

In this paper, we will need to work with $\scrExtrig_S(A,\Gm)$ instead of $\scrExtrig_S(A^0,\Gm)$, as the former
has better functorial properties due to the N\'eron mapping property of $A$ (which is not enjoyed by $A^0$).  
Following the method of Mazur-Messing \cite[I, 5.2]{MM}, we wish to show that $\scrExtrig_S(A,\Gm)$ is representable, at least as a functor on smooth test objects.
This is somewhat more subtle than the corresponding problem for $\scrExtrig_S(A^0,\Gm)$; in particular,
denoting by $\Phi_A$ and $\Phi_{\Dual{A}}$ the component groups of $A$ and $\Dual{A}$, 
we will need to know that Grothendieck's pairing for $A_K$ (see \cite[Exp. \Rmnum{7}--\Rmnum{9}]{SGA7.1} 
or \cite[\S 4]{BoschComponent})
\begin{equation}
	\xymatrix{
		{\Phi_A\times\Phi_{\Dual{A}}}\ar[r] & {\Q/\Z}
		}\label{Grpairing}
\end{equation}
is right non-degenerate.

\begin{proposition}\label{repExtrig}
	  Suppose that Grothendieck's pairing 
	  on component groups is right non-degenerate.
	  Then the fppf abelian sheaf $\scrExtrig_S(A,\Gm)$ on the category of smooth $S$-schemes is represented by
	a smooth and separated $S$-group scheme. 
	Moreover, there is a natural short exact sequence of smooth group schemes over $S$
	\begin{equation}
		\xymatrix{
			0\ar[r] & {\omega_A}\ar[r] & {\scrExtrig_S(A,\Gm)} \ar[r] & {\Dual{A}^0} \ar[r] & 0
		}.\label{extrigexact}
	\end{equation}
\end{proposition}

\begin{proof}
	We follow the proof of \cite[\Rmnum{1}, 5.2]{MM}.  Let $T$ be any smooth $S$-scheme and consider the natural map of 	abelian groups
	  \begin{equation}  
	  	\xymatrix{
			{\Extrig_T(A,\Gm)} \ar[r] & {\Ext_T(A,\Gm)}
		}.\label{forgetrig}
	  \end{equation}
	By Lemma \ref{extrigaltdesc}, we see that when $T$ is affine any extension $E$ of $A_T$ by $\Gm$
	admits a rigidification 
	  so (\ref{forgetrig}) is surjective.
	By definition, the kernel of (\ref{forgetrig}) consists of rigidifications on the trivial
	extension of $A_T$ by $\Gm$, up to isomorphism.  To give a rigidification 
	$\Inf^1_T(A_T)\rightarrow A_T\times_T \Gm$ of the trivial extension is obviously
	equivalent to giving a map of $T$-pointed $T$-schemes $\Inf^1_T(A_T)\rightarrow \Gm$,
	which in turn is equivalent to giving a global section of $\omega_{A_T}$ (cf. \cite[\Rmnum{1}, 1.2]{MM}
	or the proof of Lemma \ref{extrigaltdesc}).  If two sections $\eta_1$ and $\eta_2$ of $\omega_{A_T}$ give 
	isomorphic rigidified extensions of the trivial extension, then there is an automorphism
	of the trivial extension, necessarily induced by a group map $\varphi: A_T\rightarrow \Gm$,
	with the property that $\eta_1$ and $\eta_2$ differ by $d\varphi \in \Gamma(T,\omega_{A_T})$
	(with $d:\Hom_T(A_T,\Gm)\rightarrow \Hom(\Inf^1_T(A_T),\Gm)$ the natural map induced by the
	canonical closed immersion $\Inf^1_T(A_T)\rightarrow A_T$).  Since $A$ is flat with
	proper generic fiber and $T$ is $S$-smooth, we have $\Hom_T(A_T,\Gm)=0$ so by passing to
	the associated fppf abelian sheaves, we thus obtain the short exact sequence of fppf sheaves
	 \begin{equation*}  
	  	\xymatrix{
			0\ar[r] & {\omega_{A}}\ar[r] & {\scrExtrig_S(A,\Gm)} \ar[r] & {\scrExt_S(A,\Gm)}\ar[r] & 0
		}.
	  \end{equation*}
	 
	 Now by \cite[Proposition 5.1]{BoschComponent} (or \cite[\Rmnum{3}, Proposition C.14]{ADT}), 
	 the canonical duality of abelian varieties extends to a natural map $\Dual{A}^0\rightarrow \scrExt_S(A,\Gm)$ 
	which is an isomorphism of fppf abelian sheaves on the category of smooth $S$-schemes if and only if Grothendieck's 	pairing on component groups (\ref{Grpairing}) is right non-degenerate.  Thus,
	 our hypotheses ensure that $\scrExt_S(A,\Gm)$ is represented on the category of smooth $S$-schemes
	 by the smooth and separated $S$-group scheme $\Dual{A}^0$.  
	 Since $\omega_A$ is a vector group, it is clearly smooth and affine over $S$.
	   Thus, the proof of \cite[\Rmnum{3}, Proposition 17.4]{oort}, which is easily adapted from the situation 
	   considered there (fpqc topology on all $S$-schemes)
	 to our situation (fppf topology on smooth $S$-schemes) since $\omega_A$ and $\Dual{A}^0$ are smooth, shows
	 via fppf descent that $\scrExtrig_S(A,\Gm)$ is represented (on smooth $S$-schemes) by a smooth
	 and separated $S$-group scheme, and that there is a short exact sequence of smooth $S$-schemes (\ref{extrigexact}).
\end{proof}

\begin{remark}
 	We note that Mazur-Messing \cite[\Rmnum{1}, 5.2]{MM} prove that the canonical map
	  \begin{equation*}
		\xymatrix{
			{\Dual{A}}\ar[r] & {\scrExt_S(A^0,\Gm)}
			}
	\end{equation*}
	is an isomorphism of fppf abelian sheaves on smooth test objects for any N\'eron model $A$ 
	over any connected Dedekind scheme $S$
	by showing that  $\scrExt_S(A^0,\Gm)$ satisfies the N\'eron mapping property.  In our situation, 
	this method fails to generalize as $\Dual{A}^0$
	does not satisfy any good mapping property on smooth $S$-schemes which do not have connected closed fiber.
\end{remark}

In our applications, we will wish to apply Proposition \ref{repExtrig} when $A_K$ is the
Jacobian of a smooth and proper curve over $K$.  In this situation, it follows easily from the 
autoduality of $J_K$ and the functoriality of the morphism $\Dual{J}^0\rightarrow \scrExt_S(J,\Gm)$
that Grothendieck's pairing is right non-degenerate if and only if it is left non-degenerate if and only if it is perfect.  
In order to apply Proposition \ref{repExtrig}, we will need the following criterion for the
perfectness of Grothendieck's pairing:

\begin{proposition}\label{GrPerfJ}
	Let $X_K$ be a smooth and proper curve over $K$ with Jacobian $J_K$ over $K$.  
	Fix a proper flat and normal model $X$ of $X_K$ over $R$, and denote by $X_1,\ldots,X_n$
	the $($reduced$)$ irreducible components of the closed fiber $X_k$.  Suppose that
	the the greatest common divisor of the geometric multiplicities of the $X_i$ in $X_k$
	is 1, and assume one of the following hypotheses holds:
	\begin{enumerate}
		\item The residue field $k$ of $R$ is perfect.\label{perfectcase}
		\item $X$ is regular, each $X_i$ is geometrically reduced and $X$ admits an \'etale
		quasi-section.\label{imperfectcase}
	\end{enumerate}
	Then Grothendieck's pairing $(\ref{Grpairing})$ for $J_K$ is perfect.
\end{proposition}

\begin{proof}
	As our hypotheses are preserved by and our conclusion may be checked after \'etale base change, 
	we may replace $R$ with a strict henselization of $R$ and may thus assume that $R$ is strictly henselian.  
	In case (\ref{imperfectcase}), our hypotheses ensure that $X_K$ has a $K$-rational point and admits
	a proper flat and regular model $X$ over $R$ all of whose (reduced) irreducible components
	are geometrically reduced.  These are exactly the hypotheses of \cite[Corollary 4.7]{BoschLorenzini},
	which then ensures that Grothendieck's pairing for $J_K$ is perfect.  
	
	In case (\ref{perfectcase}), we first
	claim that our hypothesis on the gcd of the geometric multiplicities 
	of the $X_i$ in $X_k$ imply the existence of a tamely ramified Galois extension $K'$ of $K$
	(necessarily with trivial residue field extension) such that $X_{K'}$ has a $K'$-rational point.
		Indeed, by resolution of singularities for excellent surfaces \cite[\S2]{DM}, \cite{Lipman2} 
		and descent arguments from
		the completion of $R$ (see \cite[Theorem 2.2.2]{CES}) there exists
		a proper birational morphism of proper and flat $S$-models $\widetilde{X}\rightarrow X$ of 
		$X_K$ with $\widetilde{X}$ regular. 
		Due to \cite[Corollary 9.2.30]{LiuBook}, we may assume that the closed fiber $\widetilde{X}_k$ is a 
		normal crossings divisor on $\widetilde{X}$.  	
		Observe that the proper and birational morphism $\widetilde{X}\rightarrow X$ is an isomorphism 
		over any point $\xi\in X$ of codimension 1; this may be checked after the base change 
		$\Spec(\O_{X,\xi})\rightarrow X$, where it follows from the valuative criterion for properness
		applied to the discrete valuation ring $\O_{X,\xi}$ (recall that $X$ is normal).		
		In particular, $\widetilde{X}\rightarrow X$ is an isomorphism over the generic points of
		$X_k$ and we deduce that our hypothesis on the gcd
		of the geometric multiplicities of the irreducible components of $X_k$ is inherited
		by $\widetilde{X}$.  Thus, there exists an irreducible component $\Gamma_0$ of $\widetilde{X}_k$ 
		whose multiplicity $e$ in $\widetilde{X}_k$ is not divisible by $\Char(k)$.  
		The proof of \cite[Theorem 10.4.6]{LiuBook} (see also
		\cite[Corollary 10.4.7]{LiuBook}) then shows that there is a Galois extension $K'$ of $K$ with ramification
		index $e$ having the following property: 
		letting $R'$ denote the integral closure of $R$ in $K'$, 
		(which is again a discrete valuation ring, as $R$ is henselian)
		and	writing $X'$ for the normalization of
	    $\widetilde{X}\times_S \Spec(R')$, the closed fiber $X'_k$ has an irreducible component $\Gamma_0'$ over 
	    $\Gamma_0$ whose geometric multiplicity in $X'_k$ is 1; {\em i.e.} $\Gamma_0'$
	    is generically smooth.  
	    As $R'$ is strictly henselian, we conclude that there exists
	    an $R'$-point of $X'$ and hence a $K'$-point of $X'_{K'}=X_{K'}$, as claimed.	 
	
	Now since $k$ is perfect, $X_{K'}$
	admits a proper flat and regular model over $R'$ with the property that every (reduced)
	irreducible component of the closed fiber is geometrically reduced (any proper flat and regular model will do).  
	We may therefore apply
	\cite[Corollary 4.7]{BoschLorenzini} to $X_{K'}$ to deduce that Grothendieck's pairing
	for $J_{K'}$ is perfect.  As $K'/K$ is tamely ramified, it now follows from \cite{BertapelleBosch}
	that Grothendieck's pairing for $J_K$ is perfect, as desired.
\end{proof}

\begin{remark}\label{GrPerf}
	Assuming $k$ to be perfect, it follows from work of P\'epin \cite{Pepin} 
	(using the results of Bosch and Lorenzini \cite{BoschLorenzini}) 
	that Grothendieck's pairing for $J_K$ is perfect whenever the index of $X_K$ is 
	not divisible by the characteristic of $k$. 

	Already in the case of Jacobians, Grothendieck's pairing may fail to be perfect.
	Indeed, working over $R$ with imperfect residue fields, Bosch and Lorenzini
	give an explicit example of a Jacobian $J_K$ for which Grothendieck's pairing is not
	perfect \cite[Example 6.2]{BoschLorenzini}.  The first examples of abelian varieties
	for which Grothendieck's pairing is not perfect were given by \cite{BertapelleBosch}.
	
	For an arbitrary abelian variety $A_K$ over $K$,
	Grothendieck's pairing on component groups (\ref{Grpairing}) 
	is in addition known to be perfect under any of the following hypotheses:
	 \begin{enumerate}
	  	\item $R$ is of mixed characteristic $(0,p)$ and $k$ is perfect.\label{mixed}
		\item $k$ is finite.\label{kfinite}
		\item $k$ is perfect and $A_K$ has potentially multiplicative reduction.\label{potmult}
		\item There exists a tamely ramified Galois extension $K'$ of $K$ having trivial
		residue field extension such that Grothendieck's pairing for the base change $A_{K'}$
		is perfect.\label{tame}
\end{enumerate}
	For the proofs of these facts, see \cite{Begueri} in the case of
	(\ref{mixed}), \cite{McCallum} in case (\ref{kfinite}), \cite{BoschComponent} in case (\ref{potmult}),
	and \cite{BertapelleBosch} in the case of (\ref{tame}).
	See also \cite[\Rmnum{3},Theorem 2.5]{ADT}
	when $R$ has mixed characteristic and finite residue field. 
\end{remark}

We end this section by relating the group $\scrExtrig_S(A,\Gm)$ to the identity component of the canonical
extension $\E(\Dual{A}):=\scrExtrig_S(A^0,\Gm)$ of $\Dual{A}$:

\begin{lemma}\label{idencomp}
	Let $A_K$ be an abelian variety over $K$ and $A$ its N\'eron model over $R$.  Suppose
	that Grothendieck's pairing $(\ref{Grpairing})$ for $A_K$ is right non-degenerate, so 
	$\scrExtrig_S(A,\Gm)$ is a smooth $S$-group. 
	The canonical map of short exact sequences of $S$-groups
	\begin{equation}
			\xymatrix{
			0\ar[r] & {\omega_A} \ar[r]\ar[d]  & {\scrExtrig_S(A,\Gm)} \ar[r]\ar[d] & {\Dual{A}^0} \ar[r]\ar[d] & 0\\
						0\ar[r] & {\omega_A} \ar[r] & {\scrExtrig_S(A^0,\Gm)} \ar[r] & {\Dual{A}} \ar[r] & 0
					}\label{idencompmap}
	\end{equation}
	furnished from the functoriality of $\scrExtrig_S(\cdot,\Gm)$  
	by the inclusion $A^0\hookrightarrow A$ identifies $\scrExtrig_S(A,\Gm)$ with the identity component 
	of $\scrExtrig_S(A^0,\Gm)$.  
\end{lemma}

\begin{proof}
	First we claim that $\scrExtrig_S(A,\Gm)$ has connected fibers.  Since the exact sequence (\ref{extrigexact})
	commutes with base change, our claim follows from the fact that any extension of (not necessarily commutative) 
	finite type connected  group schemes over a field must be connected.  Indeed, 
	suppose that
	\begin{equation*}
		\xymatrix{
			1\ar[r] & G \ar[r] & E \ar[r] & F \ar[r] & 1
		}
	\end{equation*}
	is such extension.  Since connectedness of any scheme with a rational point is preserved by ground field extension,
	the fibers of $E\rightarrow F$ are connected as they become isomorphic to $G$ after passing to an extension 			field and $G$ is connected.
	 Thus, any separation $\{U,V\}$ of $E$ is a union of fibers
	of $E\rightarrow F$.  
	Since the quotient map $E\rightarrow F$ is faithfully flat and of finite type, it is open, so $\{U,V\}$
	is the pullback of a separation of $F$; by the connectedness of $F$ we conclude that $\{U,V\}$ is trivial 
	and $E$ is connected.
	
	To conclude, since $\scrExtrig_S(A,\Gm)$ has connected fibers 
	it suffices to show that any homomorphism from a commutative $S$-group $H$ with connected fibers to
	$\scrExtrig_S(A^0,G_m)$
	necessarily factors through $\scrExtrig_S(A,\Gm)$.  By the functoriality of $\scrExtrig_S(\cdot,\Gm)$, 
	the top row of (\ref{idencompmap}) is identified with the pullback of the bottom row along the inclusion $\Dual{A}^0\rightarrow \Dual{A}$;
	i.e. we have a canonical isomorphism of smooth groups
	$$\scrExtrig_S(A,\Gm) = \scrExtrig_S(A^0,\Gm) \times_{\Dual{A}} \Dual{A}^0.$$
	Thus, since the composition of $H\rightarrow \scrExtrig_S(A^0,\Gm)$ with the projection to $\Dual{A}$ 
	necessarily factors through the inclusion of $\Dual{A}^0$ into $\Dual{A}$
	as $H$ has connected fibers, we conclude that $H\rightarrow \scrExtrig_S(A^0,\Gm)$ indeed factors through
	the fiber product $\scrExtrig_S(A,\Gm)$, as desired.
\end{proof}

\section{An enhancement of the relative Picard functor}\label{dualizingsec}

We continue to suppose that $S=\Spec R$ with $R$ a discrete valuation ring having field of fractions $K$.  
By a {\em relative curve} $X$ over $S$ we mean a flat finite type and separated $S$-scheme $f:X\rightarrow S$ of pure relative
dimension 1 that is normal with smooth and geometrically connected generic fiber. 	
In this section, we will introduce the functor $\Pic^{\natural,0}_{X/S}$ and prove that it is representable whenever $\Pic^0_{X/S}$
is representable.   

We begin by recalling some general facts about relative dualizing sheaves and Grothendieck duality that will be needed in what follows.
Let $X$ and $Y$ be locally noetherian schemes and $f:X\rightarrow Y$ a Cohen-Macaulay morphism of pure relative dimension $n$.
By \cite[Theorem 3.5.1]{GDBC}, the complex $f^!\O_Y$ has a unique nonzero cohomology sheaf, which is in degree $-n$, and the 
{\em relative dualizing sheaf  of $X$ over $Y$} is
 \begin{equation*}
	\omega_{X/Y}:= H^{-n}(f^!\O_Y).
\end{equation*}
It is flat over $Y$ by \cite[Theorem 3.5.1]{GDBC}, and locally free if and only if the Cohen-Macaulay fibers of $f$ are Gorenstein 
\cite[V, Proposition 9.3, Theorem 9.1]{RD}.   Furthermore, the formation of $\omega_{X/Y}$ is compatible with \'etale localization on $X$ 
(see the discussion preceding \cite[Corollary 4.4.5]{GDBC} ) and with any base change $Y'\rightarrow Y$ where $Y'$ is locally noetherian 
\cite[Theorem 3.6.1]{GDBC}.  When $f$ is in addition proper, there is a natural $\O_Y$-linear trace map
\begin{equation}
	\gamma_f: R^n f_* \omega_{X/Y} \rightarrow \O_Y \label{tracemap}
\end{equation}
which is compatible with any base change $Y'\rightarrow Y$ with $Y'$ locally noetherian \cite[Corollary 3.6.6]{GDBC}.
By Grothendieck-Serre duality \cite[Theorem 4.3.1]{GDBC} the canonical map
\begin{equation}
	\xymatrix{
		{\R f_* \R\scrHom_X^{\bullet}(\F^{\bullet},\omega_{X/Y}[n])}\ar[r]  & {\R\scrHom_Y^{\bullet}(\R f_* \F^{\bullet},\O_Y)}
	},\label{GD}
\end{equation}	
induced by (\ref{tracemap}) is a quasi-isomorphism for any complex $\F^{\bullet}$ in the derived category of
sheaves of $\O_X$-modules whose cohomology is coherent and vanishes in sufficiently negative and positive degrees.

For arbitrary base schemes $Y$ and Cohen-Macaulay morphisms $f:X\rightarrow Y$ of pure relative dimension,
one defines $\omega_{X/Y}$ (and $\gamma_f$ when $f$ is proper) via direct limits and base change from the locally noetherian
case (see \cite[p. 174]{GDBC}); this makes sense due to the aforementioned base-change compatibility in the locally noetherian context
and yields a coherent sheaf of $\O_X$-modules $\omega_{X/Y}$ and a trace map $\gamma_f$ when $f$ is proper
that are compatible with arbitrary base change on $Y$.

Let us apply these considerations to the case of a relative curve $f:X\rightarrow S$.  
Since $X$ is normal and of pure relative dimension
one, it is Cohen-Macaulay by Serre's criterion for normality, so $f$ is Cohen-Macaulay
by part (\rmnum{2}) of the Corollary to Theorem 23.3 in \cite{matsumura}.  Thus, the complex $f^!\O_S$ is a coherent sheaf $\omega_{X/S}$ concentrated in   
degree -1.  By our discussion, $\omega_{X/S}$ is $S$-flat, and of formation compatible with \'etale localization on $X$ and arbitrary base change
on $S$.   When $f$ is $S$-smooth, the theory of the dualizing sheaf provides a canonical identification of the relative dualizing sheaf with the sheaf of
relative differential 1-forms on $X$ over $S$.  It is natural to ask how these two sheaves are related in general:

\begin{proposition}\label{diffmap}
	There is a canonical $\O_X$-linear morphism
	\begin{equation}
		\xymatrix{
			{c_{X/S} : \Omega^1_{X/S}} \ar[r] & {\omega_{X/S}}
		}\label{Cmap}
	\end{equation}
	whose restriction to any $S$-smooth open subset of $X$ is the canonical isomorphism.
\end{proposition}

\begin{proof}
	See \cite[Proposition 5.1]{CaisDualizing}.
\end{proof}

In fact, we can realize $\omega_{X/S}$ as a subsheaf of differentials on $X$ which are regular
on the generic fiber.  Precisely, if $i:U\hookrightarrow X$ is any open subscheme of $X$ containing the generic fiber
then the canonical map $\omega_{X/S}\rightarrow i_*i^*\omega_{X/S}$ is injective as it is an isomorphism
over $X_K$ and $\omega_{X/S}$ is $S$-flat.
Since the formation of $\omega_{X/S}$ is compatible with \'etale localization on $X$, we thus obtain a
natural injective map
\begin{equation}
	\xymatrix{
		{\omega_{X/S}} \ar@{^{(}->}[r] & {i_*\omega_{U/S}}
	}.\label{Umap}
\end{equation}
In particular, taking $U=X_K$ we have $\omega_{U/S}\simeq \Omega^1_{X_K/K}$ by 
the general theory of the dualizing sheaf (or by Proposition \ref{diffmap}),
so $\omega_{X/S}$ is a subsheaf of $i_*\Omega^1_{X_K/K}$.  When $U$ is large enough,
the map (\ref{Umap}) is also surjective:

\begin{lemma}\label{dualizing}
	Suppose that the complement of $U$ in $X$ consists of finitely many closed points of codimension $2$ $($necessarily
	in the closed fiber$)$.  Then the canonical injective map $(\ref{Umap})$ is an isomorphism. 
\end{lemma}

\begin{proof}
	We follow the proof given below (5.2.7) of \cite{GDBC}.  By standard arguments,
	it suffices to show that the local cohomology groups $H^1_x(X,\omega_{X/S})$ vanish for all $x\in X-U$.  Due to \cite[Exp. III, Example 3.4]{SGA2},
	such vanishing is equivalent to 
	$$\depth_{\O_{X,x}}(\omega_{X/S,x})\ge 2.$$
	If $x \in X-U$ is a regular point, then this inequality is trivially verified as $\omega_{X,x}$ is a free $\O_{X,x}$-module of rank 1
	for such $x$ (regular local rings are Gorenstein), and $\O_{X,x}$ is $2$-dimensional and normal (whence it has depth 2 by Serre's criterion for normality).
	
	 In general, by \cite[Exp. III, Corollary 2.5]{SGA2} it is enough to show that for each non-regular point $x$ of the closed fiber $X_k$ of $X$ we have	\begin{equation}
		\depth_{\O_{X_k,x}}(\omega_{X_k/k,x})\ge 1.\label{clsdfiberdepth}
	\end{equation}
	If this is not the case, then the maximal ideal $\m_x$ of $\O_{X_k,x}$ consists entirely of zero-divisors for the finite
	$\O_{X_k,x}$-module $\omega_{X_k/k,x}$, so it must be an associated prime of $\omega_{X_k/k,x}$.
	We would then have $\m_x=\Ann(s)$ for some nonzero $s\in\omega_{X_k/k,x}$ whence
	$\Hom_{X_k}(k(x),\omega_{X_k/k})\neq 0$.  However,
	\begin{equation}
		\Hom_{X_k}(k(x),\omega_{X_k/k}) =H^1(X_k,k(x))^{\vee}\label{kxdual}
	\end{equation}
	by Grothendieck duality for the $k$-scheme $X_k$ (see Corollary 5.1.3 and the bottom half of page 224 in \cite{GDBC}),
	and we know that the right side of (\ref{kxdual}) is zero (since $k(x)$ is a skyscraper sheaf supported at the point $x$), which
	is a contradiction.  Thus, $\m_x$ contains an $\omega_{X_k/k,x}$-regular element, so (\ref{clsdfiberdepth}) holds, as desired.
\end{proof}

When $f:X\rightarrow S$ is in addition proper, so we have a trace map (\ref{tracemap}),
we may apply the machinery of Grothendieck duality.  For our purposes, we need only the following:
			
\begin{proposition}\label{GDexplicit}
	If $f:X\rightarrow S$ is a proper relative curve then the canonical map of flat $\O_S$-modules
	\begin{equation}
		\xymatrix{
			{f_*\omega_{X/S}} \ar[r] &  {(R^1f_*\O_X)^{\vee}}
		}\label{duality1}
	\end{equation}		
	induced by cup product and the trace map $(\ref{tracemap})$ is an isomorphism.  Furthermore, there is a natural short exact 
	sequence of $\O_S$-modules
	\begin{equation}
		\xymatrix{
			0\ar[r] & {\scrExt_S^1(R^1f_*\O_X,\O_S)} \ar[r] & {R^1f_*\omega_{X/S}} \ar[r] & {(f_*\O_X)^{\vee}}\ar[r] & 0
		}.\label{duality2}
	\end{equation}
	In particular, if $f$ is cohomologically flat $($in dimension $0)$ then $R^1f_*\omega_{X/S}$
	is a locally free $\O_S$-module.
\end{proposition}	
		
\begin{proof}
	Since $\scrHom_X(\O_X,\cdot)$ is naturally isomorphic to the identity functor,
	(\ref{GD}) with $\F^{\bullet} = \O_X$ (thought of as a complex in degree zero) yields a quasi-isomorphism
	\begin{equation}
		\R f_*\omega_{X/S}[1] \simeq \R \scrHom_S^{\bullet}(\R f_*\O_{X}, \O_S).\label{qism}
	\end{equation}
	Applying $H^{-1}$ to (\ref{qism}) and using the spectral sequence
	\begin{equation}
		E_2^{m,n} = \scrExt_S^m(H^{-n}(\R f_*\O_X),\O_S) \implies H^{m+n}(\R \scrHom^{\bullet}_S(\R f_*\O_X, \O_S))\label{spectral}
	\end{equation}
	(whose only nonzero terms occur when $m=0,1$ and $n=0,-1$) to calculate the right side, we obtain a natural isomorphism 
	$f_*\omega_{X/S}\simeq {(R^1f_*\O_X)^{\vee}}$.	To know that this map coincides with the map (\ref{duality1}) induced by cup product and the trace map $\gamma_f$, one proceeds
	as in the proof of \cite[Theorem 5.1.2]{GDBC}.  
	Similarly, applying $H^0$ to (\ref{qism}) and using (\ref{spectral}), we arrive at the short exact sequence (\ref{duality2}).
	  For the final statement of the proposition, recall that by definition $f$ is cohomologically flat in dimension 0 if $f_*\O_X$ commutes with arbitrary base change,
	which holds if and only if $R^1f_*\O_X$ is locally free.  Thus, when $f$ is cohomologically flat, the sheaf $\scrExt^1_S(R^1f_*\O_X,\O_S)$ 
	vanishes and it follows easily from (\ref{duality2}) that $R^1f_*\omega_{X/S}$ is locally free over $S$.
\end{proof}

We record here the following Corollary, which shows that the relative dualizing sheaf is in general much better behaved than
the sheaf of relative differential 1-forms:

\begin{corollary}\label{locfreeBC}
	  Let $f:X\rightarrow S$ be a proper relative curve, and assume that $f$ is cohomologically flat in dimension 0.  Then
 	for all $i\ge 0$, the $\O_S$-module  $R^if_*\omega_{X/S}$ is locally free and commutes with arbitrary base change on $S$.
\end{corollary}

\begin{proof}
	By standard arguments on base change, it is enough to show that $R^if_*\omega_{X/S}$
	is locally free for $i\ge 0$.  This holds for $i\ge 2$ by the theorem on formal functions (as then $R^if_*\omega_{X/S}=0$),
	  and for $i=0$ since $\omega_{X/S}$ is $S$-flat.  For $i=1$, it follows immediately from
	  Proposition \ref{GDexplicit}. 
\end{proof}

For a relative curve $f:X\rightarrow S$, we now wish to apply the preceding considerations to define a natural enhancement $\Pic^{\natural}_{X/S}$
of the relative Picard functor classifying invertible sheaves with the additional data of a ``regular connection."  

Let $T$ be any $S$-scheme.  Since both the sheaf of relative differential 1-forms and the relative dualizing sheaf are compatible with 
base change, via pullback along $T\rightarrow S$ we obtain from (\ref{Cmap}) a natural morphism 
$\Omega^1_{X_T/T}\rightarrow \omega_{X_T/T}$, and hence an $\O_T$-linear derivation 
\begin{equation*}
	\xymatrix{
		{d_T: \O_T}\ar[r] & \omega_{X_T/T}
		}.
\end{equation*}

Fix a line bundle $\L$ on $X_T$.  Recall that a {\em connection} on $\L$ over $T$ is an $\O_T$-linear homomorphism
${\nabla:\L}\rightarrow{\L\otimes_{\O_T} \Omega^1_{X_T/T}}$ satisfying the usual Leibnitz rule.  When $X$ is not $S$-smooth,
this notion is not generally well-behaved, and it is often desirable to allow connections to have certain types of poles 
along the singularities of $X$.  For our purposes, the right notion of a connection is:

\begin{definition}\label{regconn}
	A {\em regular connection} on $\L$ over $T$ is an $\O_T$-linear homomorphism
	\begin{equation*}
		\xymatrix{
			{\nabla:\L}\ar[r] & {\L\otimes_{\O_{X_T}} \omega_{X_T/T} }
		}
	\end{equation*}	
	satisfying the Leibnitz rule: 
	$\nabla(h\eta)= \eta\otimes d_T(h)+h\nabla{\eta}$ for any sections $h$ of $\O_{X_T}$ and $\eta$ of $\L$.
	  A {\em morphism} of line bundles with regular connection over $T$ is an $\O_{X_T}$-linear morphism of the underlying line bundles
	  that is compatible with the given connections.
\end{definition}	  

\begin{remark}
	Observe that any connection $\nabla:\L\rightarrow\L\otimes \Omega^1_{X_T/T}$ on $\L$ over $T$ gives rise to a regular connection 
	on $\L$ over $T$ via composition with the map induced by $\Omega^1_{X_T/T}\rightarrow\omega_{X_T/T}$.
\end{remark}

If $\L$ and $\L'$ are two line bundles on $X_T$ equipped with regular connections $\nabla$ and $\nabla'$ over $T$,
then the tensor product $\L\otimes_{\O_{X_T}}\L'$ is naturally equipped with the tensor product regular connection $\nabla\otimes \nabla'$ induced by  decreeing
$$(\nabla \otimes\nabla')(\eta\otimes\eta'):= \eta\otimes\nabla'(\eta')+ \eta'\otimes\nabla(\eta).$$
for any sections $\eta$ of $\L$ and $\eta'$ of $\L'$. 
Observe that with respect to this operation, the pair $(\O_{X_T},d_T)$ serves as an identity element.
Thus, the set of isomorphism classes of line bundles on $X_T$ with a regular connection over $T$ has a natural abelian group structure which is obviously 
compatible with our definition of a morphism of line bundles with connection.
Furthermore, if $T'\rightarrow T$ is any morphism of $S$-schemes, then since the formation of $\omega_{X/S}$ is compatible with 
base change, any line bundle on $X_T$ with regular connection over $T$ pulls back to a line bundle on $X_{T'}$ with regular connection over $T'$.

\begin{definition}	
	Denote by $P^{\natural}_{X/S}$ the contravariant functor from the category of $S$-schemes to the category of abelian groups
	given on an $S$-scheme $T$ by
	\begin{equation*}
	 	P^{\natural}_{X/S}(T):=\left\{\parbox{6.5cm}{Isomorphism classes of pairs $(\L,\nabla)$ consisting of a line bundle $\L$ on $X_T$
		equipped with a regular connection $\nabla$ over $T$}\right\},
	\end{equation*}
	and write $\Pic_{X/S}^{\natural}$ for the fppf sheaf associated to $P^{\natural}_{X/S}$.
\end{definition}

As is customary, we will denote by $P_{X/S}$ the contravariant functor on the category of $S$-schemes which associates to an $S$-scheme $T$
the set of isomorphism classes of line bundles on $X_T$, and by $\Pic_{X/S}$ the fppf sheaf on the category of $S$-schemes associated to $P_{X/S}$.
For any $S$-scheme $T$, there is an obvious homomorphism of abelian groups $P^{\natural}_{X/S}(T) \rightarrow P_{X/S}(T)$ given 
by ``forgetting the connection," and hence a map of fppf abelian sheaves 
\begin{equation}
	\xymatrix{
		{\Pic_{X/S}^{\natural}} \ar[r] & \Pic_{X/S}
	}.\label{forget}
\end{equation}
We wish to define a certain subfunctor of $\Pic^{\natural}_{X/S}$ which will play the role of ``identity component" and which will
enjoy good representability properties.   We adopt the following definition:
\begin{definition}\label{natiden}
	Let $\Pic_{X/S}^{\natural,0}$ be the fppf abelian sheaf on the category of $S$-schemes given by
	$$\Pic_{X/S}^{\natural,0}:=\Pic^{\natural}_{X/S} \times_{\Pic_{X/S}} \Pic^0_{X/S}.$$
\end{definition}
Here, $\Pic^0_{X/S}$ denotes the identity component of the group functor $\Pic_{X/S}$ (whose fibers are representable; see
\cite[p. 459]{LiuNeron} and compare \cite[p. 233]{BLR}).  Alternately, $\Pic^0_{X/S}$ the open subfunctor of $\Pic_{X/S}$ classifying line bundles
of partial degree zero on each irreducible component of every geometric fiber \cite[\S 9.2, Corollary 13]{BLR}.

\begin{theorem}	\label{rep}
		Let $f:X\rightarrow S$ be a proper relative curve and
	suppose that the greatest common divisor of the geometric multiplicities of the irreducible components of the closed fiber $X_k$ of $X$ is 1.
	Then $\Pic^{\natural,0}_{X/S}$ is a smooth $S$-scheme and there is a short exact sequence of smooth group schemes 
		 over $S$
	\begin{equation}
		\xymatrix{
			0\ar[r] & {f_*\omega_{X/S}} \ar[r] & {\Pic^{\natural,0}_{X/S}} \ar[r] & {\Pic^0_{X/S}} \ar[r] & 0 
		}.\label{gpseq}
	\end{equation}
\end{theorem}	

To prove Theorem \ref{rep}, we will first construct (\ref{gpseq}) as an exact sequence of fppf abelian sheaves.  By work of Raynaud \cite[Theorem 8.2.1]{RaynaudPic}
(or \cite[\S9.4 Theorem 2]{BLR}),
the hypotheses on $X$ imply that $\Pic^0_{X/S}$ is a separated $S$-group scheme which is smooth by \cite[\S8.4 Proposition 2]{BLR}. 
On the other hand, our hypotheses ensure that $X$ is cohomologically flat in dimension zero, whence $f_*\omega_{X/S}$
is a vector group (in particular, it is smooth and separated) by Corollary \ref{locfreeBC}.
A straightforward descent argument will complete the proof.

We will begin by constructing the exact sequence (\ref{gpseq}).  
Fix an $S$-scheme $T$ and consider the natural map (\ref{forget}).
The kernel of this map consists of all isomorphism classes represented by pairs of the form $(\O_{X_T},\nabla)$, where $\nabla$ is a regular connection
on $\O_{X_T}$ over $T$.  By the Leibnitz rule, $\nabla$ is determined up to isomorphism by the value $\nabla(1)\in \Gamma({X_T},\omega_{X_T/T})$.
Since two pairs $(\O_{X_T},\nabla)$ and $(\O_{X_T},\nabla')$ are isomorphic precisely when there is a unit $u\in \Gamma(X_T,\O_{X_T}^{\times})$ satisfying
$$\nabla(1) = \nabla'(1)+  h^{-1}\cdot{d_Tu}$$
we see that the kernel of (\ref{forget}) is naturally identified with $H^0(X_T,\omega_{X_T/T})$ modulo
the image of the map 
\begin{equation}
	\xymatrix{
		{d_T\log:H^0(X_T,\O_{X_T}^{\times})} \ar[r] & {H^0(X_T,\omega_{X_T/T})} 
	}\label{dtlog}
\end{equation}	
that sends a global section $u$ of $\O_{X_T}$ to $u^{-1}\cdot d_Tu$.
Since pushforward by the base change ${f_T}_*$  of $f$ is left exact, we know that ${f_T}_*\O_{X_T}^{\times}$ is a subsheaf
of ${f_T}_*\O_{X_T}$.  By \cite[Th\'eor\`eme 7.2.1]{RaynaudPic}, the hypotheses on $X$ ensure that $f$ is cohomologically flat, so 
${f_T}_*\O_{X_T}\simeq \O_T$.  Since $d_T$ annihilates $\Gamma(T,\O_T)$, we conclude that  the map (\ref{dtlog}) is zero.

We thus arrive at a short exact sequence of abelian groups
\begin{equation}
	\xymatrix{
		0\ar[r] & {H^0(X_T,\omega_{X_T/T})} \ar[r] & {P^{\natural}_{X/S}(T)}\ar[r] & {P_{X/S}(T)}
	}\label{exactPic}
\end{equation}
that is easily seen to be functorial in $T$.
In order to construct the exact sequence of fppf abelian sheaves (\ref{gpseq}), we need to extend (\ref{exactPic}).
To do this, we use \v{C}ech theory to interpret (\ref{exactPic}) as part of a long exact sequence of cohomology groups. 

Consider the two-term complex (in degrees 0 and 1) $d_T\log:\O_{X_T}^{\times}\rightarrow \omega_{X_T/T}$
given by sending a section $u$ of $\O_{X_T}^{\times}$ to $u^{-1}\cdot d_Tu$; we will denote this complex by $\omega_{X_T/T}^{\times,\bullet}$.   
The evident short exact sequence of complexes
\begin{equation*}
	\xymatrix{
		0\ar[r] & {\omega_{X_T/T}[-1]} \ar[r] & {\omega_{X_T/T}^{\times,\bullet}}\ar[r] & {\O_{X_T}^{\times}} \ar[r] & 0
	}
\end{equation*}
yields (since $d_T\log:H^0(X_T,\O_{X_T}^{\times})\rightarrow H^0(X_T,\omega_{X_T/T})$ is the zero map) a long exact sequence in hypercohomology
\begin{equation}
	\xymatrix{
		0\ar[r] & {H^0(X_T,\omega_{X_T/T})}\ar[r] & {\H^1(X_T,\omega_{X_T/T}^{\times,\bullet})}\ar[r] & {H^1(X_T,\O_{X_T}^{\times})} \ar[r]^{d_T\log} &{H^1(X_T,\omega_{X_T/T})}
	}\label{lext}
\end{equation}
that is clearly functorial in $T$.   

\begin{lemma}\label{identify}
For affine $T$, the exact sequence $(\ref{exactPic})$ is identified with the first three terms of $(\ref{lext})$ in a manner that is functorial in $T$.
\end{lemma}

\begin{proof}
	Recall ({\em cf.} \cite[$0_{\mathrm{\Rmnum{3}}}$, \S 12.4]{EGA})
	that there is a natural identification of derived-functor (hyper) cohomology with \v{C}ech (hyper)cohomology which is $\delta$-functorial
	in degrees 0 and 1.  We thus have a natural identification of (\ref{lext}) with the corresponding exact sequence of \v{C}ech (hyper)cohomology
	groups, so it suffices to interpret (\ref{exactPic}) \v{C}ech-theoretically in a manner that is natural in $T$.
	
	For $(\L,\nabla)$ representing a class in $P_{X/S}^{\natural}(T)$, let $\{U_i\}$ be a Zariski open cover of $X_T$ that trivializes $\L$, and denote 
	by $f_{ij}\in \Gamma(U_i\cap U_j,\O_{X_T}^{\times})$ the transition functions.  Because of the Leibnitz rule, $\nabla\big|_{U_i}$ is determined by a 
	  unique ``connection form" $\omega_i\in \Gamma(U_i,\omega_{X_T/T})$, and the relation $$\omega_i-\omega_j = f_{ij}^{-1}\cdot d_T f_{ij}$$ holds
	on $U_i\cap U_j$.  We thus obtain a \v{C}ech 1-hyper cocycle for the complex $\omega_{X_T/T}^{\times,\bullet}$
	$$(\{f_{ij}\},\{\omega_{i}\})\in C^1(\{U_i\},\omega^{\times,\bullet}_{X_T/T}):=C^1(\{U_{i}\cap U_j\},\O_{X_T}^{\times})\oplus  C^0(\{U_i\},\omega_{X_T/T}).$$
	It is straightforward to check that any two such trivializations over open covers $\{U_i\}$ and $\{V_j\}$ yield hyper 1-cocycles which differ
	by a hyper coboundary when viewed as hyper 1-cocycles for the common refining open cover $\{U_i\cap V_j\}$, and likewise that two different 
	representatives of the same isomorphism class in $P_{X/S}^{\natural}(T)$ yield hyper 1-cocycles that differ by a hyper-coboundary (after passing
	to a common refining cover of the associated cocycles). 
	We therefore obtain a well-defined \v{C}ech hyper-cohomology class.  This procedure is easily reversed, and so we have a bijection 
	$$P^{\natural}_{X/S}(T)\simeq \check{\H}^1(X_T,\omega_{X_T/T}^{\times,\bullet}).$$
	To check that this is in fact a homomorphism of abelian groups that is functorial in $T$ is straightforward (albeit tedious).
	
	We identify $P_{X/S}(T)$ with $\check{H}^1(X_T,\O_{X_T}^{\times})$ in the usual way, by sending a class represented by $\L$ to the 1-cocycle 
	$\{f_{ij}\}$ given by the transition functions associated to a trivializing open cover $\{U_i\}$ and choice of trivializations of $\L\big|_{U_i}$.
	Similarly, we use the natural isomorphism of $H^0(X_T,\omega_{X_T/T})$  with $\check{H}^0(X_T,\omega_{X_T/T})$, and we thus obtain
	a functorial diagram of homomorphisms of abelian groups  
	\begin{equation*}
		\xymatrix{
			0\ar[r] & {H^0(X_T,\omega_{X_T/T})} \ar[r]\ar[d]^-{\simeq} & {P^{\natural}_{X/S}(T)}\ar[r]\ar[d]^-{\simeq} & {P_{X/S}(T)}\ar[d]^-{\simeq}\\
			0\ar[r] & {\check{H}^0(X_T,\omega_{X_T/T})}\ar[r] & {\check{\H}^1(X_T,\omega_{X_T/T}^{\times,\bullet})}\ar[r] & {\check{H}^1(X_T,\O_{X_T}^{\times})}
		}
	\end{equation*}
	That this diagram commutes is easily verified by appealing to the explicit descriptions of the maps involved.

\end{proof}

By Raynaud's ``crit\`ere de platitude cohomologique" \cite[Th\'eor\`eme 7.2.1]{RaynaudPic}, our hypotheses $X$ ensure that $f$ is cohomologically flat in dimension zero.
Thus,  due to Corollary \ref{locfreeBC} and the fact that the formation of $\omega_{X/S}$ commutes with any base change on $S$,
for each $i\ge 0$ the fppf sheaf associated to functor on $S$-schemes
\begin{equation*}
	T\rightsquigarrow H^i(X_T,\omega_{X_T/T})
\end{equation*}
is represented by the vector group $R^if_*\omega_{X/S}$.  By Lemma \ref{identify}, we therefore have an exact sequence of fppf sheaves
of abelian groups on the category of $S$-schemes whose first and last (nonzero) terms are smooth affine $S$-groups:
\begin{equation}
	\xymatrix{
		0\ar[r] & {f_*\omega_{X/S}}\ar[r] & {\Pic^{\natural}_{X/S}}\ar[r] & {\Pic_{X/S}} \ar[r] & {R^1f_*\omega_{X/S}}
	}\label{close}
\end{equation}

With (\ref{close}) at hand, we can now prove Theorem \ref{rep}:

\begin{proof}[Proof of Theorem \ref{rep}]
	Consider the identity component $\Pic^0_{X/S}$ and the composition of its inclusion into $\Pic_{X/S}$ with the map
	of fppf sheaves $\Pic_{X/S}\rightarrow R^1f_*\omega_{X/S}$.  We claim that this composite is the zero map.  Indeed,
	by Exemples 6.1.6 and Th\'eor\`eme 8.2.1 of \cite{RaynaudPic} (or \S9.4, Theorem 2 of \cite{BLR}) and \S 8.4, Proposition 2 of \cite{BLR}, 
	  our hypotheses on $X$ ensure that $\Pic^0_{X/S}$ 
	is a smooth and separated $S$-scheme, so the composite map 
	\begin{equation}
		\xymatrix{
			{\Pic^0_{X/S}}\ar[r] & {R^1f_*\omega_{X/S}}
			}\label{Sgpmap}
	\end{equation}
	is a map of $S$-group schemes.
	Since the generic fiber of $\Pic^0_{X/S}$ is an abelian variety and $R^1f_*\omega_{X/S}$ is affine over $S$,
	the closed kernel of (\ref{Sgpmap}) contains the generic fiber, and hence
	(\ref{Sgpmap}) is the zero map.  Thus, the inclusion 
	$\Pic^0_{X/S}\rightarrow \Pic_{X/S}$ factors through the image of (\ref{forget}).
	By pullback, we obtain a short exact sequence (\ref{gpseq}) of fppf abelian sheaves on the category of $S$-schemes.
	As we have observed, the leftmost term in (\ref{gpseq}) is a vector group (in particular it is a smooth and affine $S$-group), and the rightmost term
	is a smooth and separated $S$-group scheme.  It follows from this by fppf descent, as in the proof of Proposition \ref{repExtrig}, that
	$\Pic^{\natural,0}_{X/S}$ is a smooth and separated $S$-group scheme, and that we have a short exact sequence (\ref{gpseq}) of smooth
	and separated group schemes over $S$.
\end{proof}

\section{Proof of the main Theorem}\label{mainthmpf}

In this section, we prove Theorem \ref{main}, following the outline sketched in the introduction (in particular,
we will keep our notation from that section).
Throughout this section, we fix a proper relative curve $f:X\rightarrow S$ over $S=\Spec R$
which we suppose satisfies the hypotheses of Theorem \ref{main}.
Note that these hypotheses ensure that Grothendieck's pairing on component groups for $J_K$
is perfect, by Proposition \ref{GrPerfJ}. 
In particular, there is a natural short exact sequence of smooth $S$-groups:
\begin{equation*}
	\xymatrix{
			0\ar[r] & {\omega_J}\ar[r] & {\scrExtrig_S(J,\Gm)} \ar[r] & {\Dual{J}^0} \ar[r] & 0
			}.
\end{equation*}

We begin our proof of Theorem \ref{main} by 
constructing a canonical map of short exact sequences of smooth $S$-group schemes
\begin{equation}
	\xymatrix{
		0\ar[r] & {\omega_J} \ar[d]\ar[r] & {\scrExtrig_S(J,\Gm)} \ar[r]\ar[d] & {\Dual{J}^0} \ar[r]\ar[d] & 0\\
			0\ar[r] & {f_*\omega_{X/S}} \ar[r] & {\Pic^{\natural,0}_{X/S}} \ar[r] & {\Pic^0_{X/S}}\ar[r] & 0
	}\label{can}
\end{equation}
which we do in 3 steps. 
	
\medskip\noindent
{\bf Step 1:}  We initially suppose that there exists a rational point $x\in X_K(K)$ and will later explain how to reduce the general case to this one.
Associated to $x$ is the usual Albanese mapping ${j_{x,K}}: X_K\rightarrow J_K$ given by the functorial recipe ``$y\mapsto \O(y)\otimes \O(x)^{-1}$."
Letting $i:X^{\sm}\hookrightarrow X$ denote the $S$-smooth locus of $f:X\rightarrow S$ in $X$, we denote by 
$j_x:X^{\sm}\rightarrow J$ the morphism obtained from 
$j_{x,K}$ by the N\'eron mapping property of $J$.  By abuse of notation, we will also write $j_x$ for any base change of $j_x$.	
For each smooth and affine $S$-scheme $T$, we will show that ``pullback along $j_x$" yields a commutative diagram of exact sequences
of abelian groups
\begin{equation}
	\xymatrix{
		0\ar[r] & {\Gamma(T,\omega_{J_T})} \ar[r]\ar[d] & {\Extrig_T(J_T,\Gm)} \ar[r]\ar[d] & {\Ext_T(J_T,\Gm)} \ar[r]\ar[d] & 0\\
		0\ar[r] & {\Gamma(X_T,\omega_{X_T/T})} \ar[r] & {P_{X/S}^{\natural}(T)} \ar[r] & {P_{X/S}(T)}
	}\label{step1}
\end{equation}
that is functorial in $T$. 
To do this, we will need to apply the following Lemma with $U=X^{\sm}$; that this choice of $U$ satisfies
the hypotheses of the Lemma crucially uses our hypothesis that the closed fiber of $X$ is generically smooth.
\begin{lemma}\label{reduce2U}
	Let $U$ be any open subscheme of $X$ whose complement in $X$ consists of points of codimension at least 2.
	For each smooth $S$-scheme $T$, pushforward along $i_T:U_T\rightarrow X_T$ yields a natural isomorphism 
	of short exact sequences
	of abelian groups
	\begin{equation*}
		\xymatrix{
			0\ar[r] & {\Gamma(U_T,\Omega^1_{U_T/T})} \ar[r]\ar[d]^-{\simeq} & {P_{U/S}^{\natural}(T)}\ar[r]\ar[d]^-{\simeq} & {P_{U/S}(T)}\ar[d]^-{\simeq}\\
			0\ar[r] & {\Gamma(X_T,\omega_{X_T/T})} \ar[r] & {P_{X/S}^{\natural}(T)}\ar[r] & {P_{X/S}(T)}
			}.
	\end{equation*}
\end{lemma}

\begin{proof}
To minimize notation, we will simply write $i$ for $i_T$.
Since the dualizing sheaf is compatible with \'etale localization,
it suffices to show that for any pair $(\L,\nabla)$ consisting of a line $\L$ bundle on $X_T$ with 
regular connection $\nabla$ over $T$,
the canonical commutative diagram 
\begin{equation}
	\xymatrix{
		{\L} \ar[d]_-{\nabla} \ar[r] & { i_*i^*\L} \ar[d]^-{i_* i^*(\nabla)}\\
		{ \L \otimes\omega_{X_T/T}} \ar[r] & {i_*i^*\L\otimes i_*i^*\omega_{X_T/T}}
	}\label{restr}
\end{equation}
has horizontal arrows that are isomorphisms.  By hypothesis, $X$ is normal and the complement of $U$ in $X$ consists of points of codimension at least two.  
Since $T\rightarrow S$ is smooth, the base change $X_T$ is also normal and the complement of $U_T$ in $X_T$ has codimension
at least 2 (see part (ii) of the Corollary to Theorem 23.9 and Theorem 15.1 in \cite{matsumura}).
As $\L$ is locally free, it follows that the top horizontal map of (\ref{restr}) is an isomorphism.
By Lemma \ref{dualizing} the canonical map $\omega_{X/S}\rightarrow i_*i^*\omega_{X/S}$ is an isomorphism; 
since this map and the sheaves in question are compatible with base change, we conclude that the bottom horizontal 
arrow in (\ref{restr}) is also an isomorphism.
\end{proof}

\begin{remark}\label{reduce2Ufalse}
	Note that Lemma \ref{reduce2U} is generally false if the complement of $U$ in $X$ has codimension
	strictly less than 2.
\end{remark}

We deduce from Lemma \ref{reduce2U} applied to $U:=X^{\sm}$ that it suffices to construct (\ref{step1}) with $X$ 
replaced by $U$ in the bottom row.
Note that since $U_T$ is $T$-smooth, the notions of regular connection and connection coincide (see Proposition \ref{diffmap}).
Thus, we wish to associate to any element of $\Extrig(J_T,\Gm)$ an invertible sheaf on $U_T$ with connection over $T$ in
a manner that is Zariski-local on (and functorial in) $T$, and so globalizes from the case of affine $T$.  To do this, 
we proceed as follows.

Fix a choice $\tau$ of generator for $\omega_{\Gm}$ and (functorially) identify $\Extrig_T(J_T,\Gm)$ with $\mathbf{E}_{\tau}(J_T)(T)$ via Lemma \ref{extrigaltdesc}.
Let $(E,\eta)$ be a representative of a class in $\mathbf{E}_{\tau}(J_T)(T)$.  Viewing $E$ as
a ${\Gm}$-torsor over $J_T$, we choose a Zariski open cover $\{V_i\}$ of $J_T$ and 
local sections $s_i:V_i\rightarrow E$ to the projection $E\rightarrow J_T$ that trivialize $E$ over $V_i$.  
Set $\omega_i:=s_i^*\eta\in \Gamma(V_i,\Omega^1_{J_T/T})$ and let $\L$ be the invertible sheaf on $J_T$
corresponding to the $\Gm$-torsor $E$.  There are two canonical ways to associate transition
functions to $\L$ and the sections $s_i$ depending on whether we consider the section $s_i-s_j: V_i\cap V_j\rightarrow \Gm$ 
or its inverse $s_j-s_i$.  However, since any two choices of $\tau$ differ by multiplication by $\pm 1$,
there is a {\em unique} choice $f_{ij}:V_i\cap V_j\rightarrow \Gm$ with the property
that $f_{ij}^*\tau = f_{ij}^{-1} d f_{ij}$ (interpreting $f_{ij}$ as a section of $\Gm$ over $V_i\cap V_j$), and we consistently make this choice of transition function.

Define
\begin{equation*} 
	\xymatrix{
		{\nabla_i: \L\big|_{V_i}} \ar[r] & {\L\big|_{V_i} \otimes_{\O_{V_i}} \Omega^1_{V_i/T}}
	}
\end{equation*}
by $\nabla_i(t s_i):= ts_i\otimes \omega_i + s_i\otimes dt$ for any section $t$ of $\O_{V_i}$. 
Using the definition of $\omega_i$ and the fact that $\eta$ pulls back to $\tau$ on $\Gm$, it is straightforward to check that 
$$\omega_i-\omega_j =f_{ij}^*\tau= f_{ij}^{-1} df_{ij}$$
(by our choice of $f_{ij}$)
in $\Gamma(V_i\cap V_j,\Omega^1_{J_T/T})$ and hence that the $\nabla_i$ 
uniquely glue to give a connection $\nabla$ on $\L$ over $T$.  By passing to a common refining open cover, one checks that any other choice of 
trivialization $(V'_{i'},s'_{i'})$ yields the same connection on $\L$, so the pair $(\L,\nabla)$ is independent of our choices of cover $\{V_i\}$ and sections $\{s_i\}$

By pullback along $j_x:U_T\rightarrow J_T$, we thus obtain a line bundle on $U_T$ with a connection.  
If $(E',\eta')$ is another choice of representative for the same class in $\mathbf{E}_{\tau}(J_T)(T)$ then
by definition there is an isomorphism of extensions $\varphi:E\rightarrow E'$
with the property that $\varphi^*\eta'=\eta$.  One easily checks that $\varphi$ induces an isomorphism
of the invertible sheaves on $J_T$ with connection corresponding to $(E,\eta)$ and to $(E',\eta')$, and hence   
that we have a well-defined map of sets ${\mathbf{E}_{\tau}(J_T)(T)} \rightarrow  {P_{U/S}^{\natural}(T)}$
which is readily seen to be functorial in $T$.  
That this map is in fact a homomorphism of abelian groups follows easily from the definition using the description of the group law
on $\mathbf{E}_{\tau}(J_T)(T)$ as in \S\ref{canext}.    By Lemma \ref{extrigaltdesc}, we  thus obtain  a homomorphism of abelian groups
\begin{equation}
	\xymatrix{
		{\Extrig_T(J_T,\Gm)} \ar[r] &  {P_{U/S}^{\natural}(T)}
	}\label{extrigtoP}
\end{equation}
that is functorial in $T$.  It is straightforward to check that this map is moreover independent of our initial choice of $\tau$ (but may {\em a priori}
depend on our choice of rational point $x$) and so provides the desired functorial map.

We similarly define $\Ext_T(J_T,\Gm)\rightarrow P_{U/S}(T)$ by associating to an extension $E$ of $J_T$ by $\Gm$
the pull back along $j_x:U_T\rightarrow J_T$ of the line bundle $\L$ on $J_T$ obtained by viewing $E$ as a $\Gm$-torsor
over $J_T$.  This is readily seen to be a homomorphism of abelian groups (using Baer sum on $\Ext_T(J_T,\Gm)$) and is obviously functorial in $T$.

Finally, we define $\Gamma(T,\omega_{J_T})\rightarrow \Gamma(U_T,\Omega^1_{U_T/T})$ as follows.
By \cite[\S4.1, Proposition 1]{BLR}, any global section $\omega_0$ of $\omega_{J_T}= e_T^*\Omega^1_{J_T/T}$ can be uniquely propagated
to an invariant differential form $\omega$ on $J_T$ over $T$ satisfying
$e_T^*\omega=\omega_0$.  Pulling $\omega$ back along $j_x:U_T\rightarrow J_T$, we obtain a section
of $\Omega^1_{U_T/T}$ over $U_T$.  This association is clearly a homomorphism and functorial in $T$. 

We thus obtain (via Lemma \ref{reduce2U}) a diagram (\ref{step1}) with all maps homomorphsms of abelian groups, functorially in smooth affine $S$-schemes $T$.  
That this diagram commutes follows immediately from the explicit definition of all the maps involved (morally, each vertical map is simply ``pullback by $j_x$").

\medskip\noindent
{\bf Step 2:}  Passing from (\ref{step1}) to the corresponding diagram of associated fppf sheaves and recalling the
construction of the exact sequence of fppf sheaves (\ref{close}) in \S \ref{dualizingsec}, we obtain a commutative diagram
of fppf sheaves of abelian groups 
\begin{equation*}
		\xymatrix{
			0 \ar[r] & {\omega_{J}} \ar[r]\ar[d] & {\scrExtrig_S(J,\Gm)} \ar[r]\ar[d] & {\scrExt_S(J,\Gm)} \ar[r]\ar[d] & 0\\
			0 \ar[r] & {f_*\omega_{X/S}} \ar[r] & {\Pic^{\natural}_{X/S}} \ar[r] & {\Pic_{X/S}}
		}.
\end{equation*}
From Lemma \ref{GrPerfJ}, we thus
deduce the following commutative diagram of fppf abelian sheaves on smooth $S$-schemes with each term in the top row a smooth $S$-group:
\begin{equation*}
		\xymatrix{
			0 \ar[r] & {\omega_{J}} \ar[r]\ar[d] & {\scrExtrig_S(J,\Gm)} \ar[r]\ar[d] & {\Dual{J}^0} \ar[r]\ar[d] & 0\\
			0 \ar[r] & {f_*\omega_{X/S}} \ar[r] & {\Pic^{\natural}_{X/S}} \ar[r] & {\Pic_{X/S}}
		}
\end{equation*}
Since the map $\Dual{J}^0\rightarrow \Pic_{X/S}$ is homomorphism of group functors (on smooth $S$-schemes) and $\Dual{J}^0$ has connected fibers, 
for topological reasons this map necessarily factors through the open subfunctor
$\Pic^0_{X/S}$ (thinking of $\Pic^0_{X/S}$ as the union of all identity components of the fibers of $\Pic_{X/S}$; {\em cf}. \cite[$\mathrm{\Rmnum{4}}_3$, 15.6.5]{EGA} and arguing
fiber-by-fiber).  By the definition of $\Pic^{\natural,0}_{X/S}$ (Definition \ref{natiden}) as a fiber product, we thus have a commutative diagram
of fppf abelian sheaves on smooth $S$-schemes
\begin{equation}
		\xymatrix{
			0 \ar[r] & {\omega_{J}} \ar[r]\ar[d] & {\scrExtrig_S(J,\Gm)} \ar[r]\ar[d] & {\Dual{J}^0} \ar[r]\ar[d] & 0\\
			0 \ar[r] & {f_*\omega_{X/S}} \ar[r] & {\Pic^{\natural,0}_{X/S}} \ar[r] & {\Pic^0_{X/S}}\ar[r] & 0
		}\label{almost}
\end{equation}

\medskip\noindent
{\bf Step 3:}  By Proposition \ref{repExtrig} and Theorem \ref{rep}, both rows of (\ref{almost}) are short exact sequences of smooth
$S$-group schemes, and we claim that the comutative diagram (\ref{almost}) of fppf sheaves on smooth $S$-schemes can be enhanced to a corresponding commutative diagram of maps between smooth group schemes over $S$.
Indeed, this follows from Yoneda's lemma, which ensures that the natural ``restriction to the smooth site" map
$$\Hom_S(\F,\G)\rightarrow \Hom_{S_{\sm}}(\F,\G)$$
is bijective for any fppf abelian sheaves $\F$, $\G$ on $S$-schemes with $\F$ represented by a smooth $S$-group scheme.

We have thus constructed (\ref{can}) using our initial choice of rational point $x\in X_K(K)$.  If $x'$ is any other choice of rational point,
we claim that the resulting maps (\ref{can}) obtained from $x$ and $x'$ coincide.  Since $j_{x,K},j_{x',K}: X_K\rightarrow J_K$ differ
by a translation  on $J_K$, it is enough to show that for any translation $\tau:J_K\rightarrow J_K$, the induced map
\begin{equation}
	\xymatrix{
		0\ar[r] & {\omega_J}\ar[r]\ar[d]^-{\varphi_1} & {\scrExtrig_S(J,\Gm)}\ar[r]\ar[d]^-{\varphi_2} & {\Dual{J}^0} \ar[r]\ar[d]^-{\varphi_3}& 0\\
		0\ar[r] & {\omega_J}\ar[r] & {\scrExtrig_S(J,\Gm)}\ar[r] & {\Dual{J}^0} \ar[r] & 0
	}\label{transmap}
\end{equation}
(using the N\'eron mapping property) is the identity map.  Since each term in the bottom row is separated and each term in the top
row is flat, whether or not (\ref{transmap}) coincides with the identity map may be checked on generic fibers.
Now $\tau^*:\omega_{J_K}\rightarrow \omega_{J_K}$ is the identity map as $\omega_{J_K}$ is identified with the sheaf of (translation) 
invariant differentials \cite[\S4.2 Proposition 1]{BLR}.
That $\tau^*: \Dual{J}_K\rightarrow \Dual{J}_K$ is the identity is well-known, and follows from the fact that the line-bundles classified by
$\Dual{J}_K:=\Pic^0_{J_K/K}$ are translation invariant (or equivalently that the classes in $\scrExt_{K}(J_K,\Gm)$ are {\em primitive}).\footnote{More precisely, for any abelian variety $A$ over $K$ we have a homomorphism of group functors 
$$\phi:\Pic_{A/K}\rightarrow \mathbf{Hom}(A,\Pic^0_{A/K})$$ given functorially on $K$-schemes $T$ by sending a line bundle $\L$ on $A_T$
to the map $x\mapsto \tau_x^*\L\otimes \L^{-1}$ with $\tau_x$ translation by a $T$-point $x$ of $A_T$. 
Since $A$ and $\Pic^0_{A/K}$ are projective, Grothendieck's theory of Hom-schemes ensures
that $\mathbf{Hom}(A,\Pic^0_{A/K})$ is a finite-type $K$-scheme which we claim is \'etale.  
Working over $\Kbar$, our claim follows from the fact that there are no nonzero liftings to $K[\epsilon]$
of the zero map $A\rightarrow \Pic^0_{A/K}$ (due to \cite[Theorem 6.1]{GIT}), so the tangent space
of $\mathbf{Hom}(A,\Pic^0_{A/K})$ at the origin is zero.  Again passing to $\Kbar$, we conclude that
the group map $\phi$ maps connected components of $\Pic_{A/K}$ to individual points, so in particular
restricts to the zero map on the connected component of the identity $\Pic^0_{X/K}$.}
Thus $\varphi_1=\varphi_3=\id$ so on $K$-fibers, $\varphi_3-\id$ uniquely factors through a map $\Dual{J}_K \rightarrow \omega_{J_K}$
which takes the identity to the identity. 
As any map from an abelian variety to a vector group is constant, we conclude that $\varphi_3-\id$ is identically zero on $K$-fibers, and hence
that $\varphi_3=\id$ as well.  
Thus, the map (\ref{can}) which we have constructed is independent of the choice of rational point $x\in X_K(K)$.

In the general case when $X_K(K)$ may be empty, we proceed as follows.  Denote by $Y$ any one of the three
schemes occurring in the top row of (\ref{can}) and by $Z$ the corresponding scheme in the bottom row.
We first claim that we have a canonical map $Y_K\rightarrow Z_K$.  Indeed, $X_K$ has a $K'$-rational point for some
finite Galois extension $K'$ of $K$, and we may use this point to define a $K'$-map 
$Y_{K'}\rightarrow Z_{K'}$ as we have explained.  Since this map is independent of the choice of $K'$-rational point by what we have said above, via Galois descent we have a canonical $K$-map $\varphi_K:Y_K\rightarrow Z_K$ as claimed.

We now appeal to the following general Lemma: 

\begin{lemma}\label{fpqctrick}
	Fix an integral scheme $T$ with generic point $\eta$ and
	let $Y \rightarrow T$ and $Z\rightarrow T$ be any flat $T$-schemes, with $Z$ separated over $T$.
	Suppose given a map $\varphi_{\eta}: Y_{\eta}\rightarrow Z_{\eta}$. 
	Then there is at most one extension of $\varphi_{\eta}$
	to a $T$-map $\varphi:Y\rightarrow Z$, and $\varphi$ exists if and only if the schematic closure in $Y\times_T Z$
	of the graph of $\varphi_{\eta}$ maps isomorphically onto $Y$ by the first projection.
	In particular, $\varphi$ exists if and only if there is an fpqc morphism $T'\rightarrow T$
	and a map $\varphi':Y_{T'}\rightarrow Z_{T'}$ with  $\varphi'_{\eta'}=\varphi_{\eta'}$
	where $\eta'=T'\times_{T} \eta$.
\end{lemma}

\begin{proof}
	The proof of Lemma \ref{fpqctrick} proceeds via standard arguments with schematic closures of graphs; due to lack 
	of a reference, we sketch how this goes.  The uniqueness of an extension 
		is clear, as $T$ is integral, $Z$ is separated over $T$, and $Y$ is $T$-flat.
	For existence, we proceed as follows.  Let $\Gamma\subseteq Y\times_T Z$ be the schematic closure
	in $Y\times_T Z$ of the graph $\Gamma_{\varphi_{\eta}}\subseteq Y_{\eta}\times_{\eta} Z_{\eta}$ of $\varphi_{\eta}$,
	and note that $\Gamma_{\eta}=\Gamma_{\varphi_{\eta}}$ as $Z$ is $T$-separated.
	Now if the first projection $\Gamma\rightarrow Y$ is an isomorphism then it is clear
	that $\varphi_{\eta}$ extends to a $T$-morphism.
	Conversely, given $\varphi:Y\rightarrow Z$ extending $\varphi_{\eta}$ and denoting by $\Gamma_{\varphi}$
	the graph of $\varphi$, we claim that necessarily $\Gamma=\Gamma_{\varphi}$.  Indeed,
	the canonical closed immersion $\Gamma\rightarrow Y\times_T Z$
	factors through a closed immersion $\Gamma\rightarrow \Gamma_{\varphi}$ as $\Gamma_{\varphi}$ is closed
	in $Y\times_T Z$ (due to $T$-separatedness of $Z$) and contains $\Gamma_{\varphi_{\eta}}$.	
	Since the closed immersion
	$\Gamma\rightarrow \Gamma_{\varphi}$ is an isomorphism over $\eta$ 
	(using that $\Gamma_{\eta}\simeq \Gamma_{\varphi_{\eta}}$) it must be an isomorphism, since
	$\Gamma_{\varphi_\eta}$ is dense in $\Gamma_{\varphi}$ as $\Gamma_{\varphi}$ is $T$-flat
	and $T$ is integral.
 	We conclude that $\Gamma=\Gamma_{\varphi}$ maps isomorphically onto $Y$ via the first projection.
	Finally, whether or not the first projection $\Gamma\rightarrow Y$ is an isomorphism is insensitive
	to fpqc base change; since the formation of $\Gamma$ commutes with such base change (as 
	$\eta\rightarrow T$ is quasi-compact and separated), we deduce
	the last statement of the Lemma.
\end{proof}

Applying the Lemma with $T=S=\Spec(R)$, $Y$, $Z$ as above, and $T'=\Spec(R^{\sh})$ for $R^{\sh}$ a strict henselization of
$R$, we see that it remains to construct a $T'$-morphism $Y_{T'}\rightarrow Z_{T'}$ recovering the base change
of $\varphi_{K}$ to $K^{\sh}:=\Frac(R^{\sh})$ on generic fibers.  
Now $X$ has generically smooth closed fiber, so $X_K$ has a $K^{\sh}$-point.
As our hypotheses on $X$ are unaltered by base change along  local-\'etale
extensions of discrete valuation rings (such as $R\rightarrow R^{\sh}$)
and the formation of the top and bottom rows of (\ref{can}) commute with such base change
we may use this $K^{\sh}$-point as in the construction of (\ref{can}) to define 
the desired $T'$-map $Y_{T'}\rightarrow Z_{T'}$. We conclude that $\varphi_K$ uniquely extends to an $S$-map,
and thus we obtain (\ref{can}) over $S$ as desired.

Now that we have constructed the canonical map of short exact sequences of smooth $S$-groups (\ref{can}), we can 
show that it is
an isomorphism.  We reiterate here that only the {\em construction} of this map uses the hypothesis that the closed fiber
of $X$ is generically smooth; as we will see below, the proof that (\ref{can}) is an isomorphism requires only the weaker hypotheses of
Raynaud's Theorem \ref{Rthm}.

\begin{proof}[Proof of Theorem $\ref{main}$]
	  
	  By passing to a finite \'etale extension if necessary, we may assume that $X_K(K)$ is nonempty, and we select $x\in X_K(K)$
	and use it to define (\ref{can}).  Note that since $X$ is normal with generically smooth closed fiber, $X$ satisfies the hypotheses
	of Theorem \ref{Rthm}.
 
	Consider the composite mapping
	\begin{equation}
		\xymatrix{
			{\Pic^0_{X/S}} \ar[r] & {J^0} \ar[r]^-{\simeq} & {\Dual{J}^0}
		}\label{invmap}
	\end{equation}  
	where the first map is deduced via the N\'eron mapping property from the canonical identification $J_K=\Pic^0_{X_K/K}$
	and the second map is similarly obtained from the canonical principal polarization $J_K\rightarrow \Dual{J}_K$
	induced by the $\Theta$-divisor; {\em cf.} \cite[\S 6]{MilneJacobians}.  
	  We claim that the composite $\Dual{J}^0\rightarrow \Dual{J}^0$ 
	  of (\ref{invmap}) with the right vertical map of (\ref{can}) coincides with negation on $\Dual{J}^0$.  
	  Since $\Dual{J}^0$ is flat and separated, it suffices to check this claim on generic fibers, so we
	wish to show that the map $\Pic^0(j_{x,K}): \Dual{J}_K\rightarrow J_K$ is the negative of the inverse of
	the canoincal principal polarization $J_K\rightarrow \Dual{J}_K$.  This is the content of \cite[Lemma 6.9]{MilneJacobians}.
	It follows from Theorem \ref{Rthm} that the right vertical map of (\ref{can}) is an isomorphism if and only if $X$ has rational 
	singularities; in particular, this settles the ``only if" direction of Theorem \ref{main}.
	
	We henceforth suppose that $X$ has rational singularities and we wish to show that (\ref{can})
	is an isomorphism of exact sequences of smooth group schemes over $S$. 
	By Theorem \ref{Rthm} and our discussion, we know that the right vertical map of (\ref{can}) is an isomorphism,
	and we will ``bootstrap" Raynaud's theorem using duality; more precisely, we will show that the left vertical map
	in (\ref{can}) is dual to the map on Lie algebras obtained from (\ref{Rthmmap}) and must therefore be an isomorphism
	as well.
	 Indeed, consider the dual of the map on Lie algebras obtained from (\ref{Rthmmap}):
	 $$\xymatrix@1{{\scrLie(J^0)^{\dual}}\ar[r]^-{\simeq} & {\scrLie(\Pic^0_{X/S})^{\vee}}}.$$
	 For any commutative group functor $G$ over $S$ with representable fibers, the canonical inclusion $G^0\hookrightarrow G$
	 induces an isomorphism on Lie algebras \cite[Proposition 1.1 (d)]{LiuNeron}, so we obtain a natural isomorphism of $\O_S$-modules 
	 $\scrLie(J)^{\vee}\simeq \scrLie(\Pic_{X/S})^{\vee}$.  The canonical identification $R^1f_*\O_X\simeq \scrLie(\Pic_{X/S})$
	 (\cite[\S8.4, Theorem 1]{BLR} or \cite[Proposition 1.3 (b)]{LiuNeron}) then gives a natural isomorphism 
	 \begin{equation}
	 	\xymatrix{
			{\scrLie(J)^{\vee}} \ar[r]^-{\simeq} & {(R^1f_*\O_X)^{\vee}}
		}.\label{dualofRthm}
	 \end{equation}
	Using the canonical duality $\omega_J\simeq \scrLie(J)^{\vee}$ (see \cite[II, \S4.11]{SGA3vol1} or \cite[Proposition 1.1 (b)]{LiuNeron})
	and Grothendieck duality (Proposition \ref{GDexplicit}) yields a natural isomorphism of $\O_S$-modules
	\begin{equation}
		\xymatrix{
			{\omega_J} \ar[r]^-{\simeq} & {\scrLie(J)^{\dual} }\ar[r]^-{\simeq}_-{(\ref{dualofRthm})} & {(R^1f_*\O_X)^{\vee}} &  \ar[l]_-{\simeq}^-{(\ref{duality1})} { f_*\omega_{X/S}}		
		}\label{Rthmtrick}
	\end{equation}
    	and hence an isomorphism of the corresponding vector groups over $S$.  We claim that the left vertical map in (\ref{can}) coincides with the negative of
	(\ref{Rthmtrick}).  Since the source of both maps is flat and the target is separated over $S$, it suffices to check such agreement
	on generic fibers.  
	
	To do this, we consider the following diagram, in which we simply write $j$ for $j_{x,K}$
	and $\varphi:J_K\rightarrow \Dual{J}_K$ for the canonical principal polarization:
	\begin{equation}
			\xymatrix{
				{\Gamma(\Spec K,\omega_{J_K})}\ar[d]_-{\text{can}}^-{\simeq} \ar[r]_-{\Pic^0(j)^*}  & 
				{\Gamma(\Spec K,\omega_{\Dual{J}_K})} \ar[d]_-{\text{can}}^-{\simeq}
				\ar[r]_-{\text{ev}}^-{\simeq}  & {\Lie(\Dual{J}_K)^{\dual}} \ar[d]_-{\theta_{J_K}^{\dual}}^-{\simeq}   
				\ar[r]_-{\Lie(-\varphi)^{\dual}}^-{\simeq} & {\Lie(J_K)^{\dual}}\ar[d]_-{\theta_{X_K}^{\dual}}^-{\simeq}\\
					  {H^0(J_K,\Omega^1_{J_K/K})} & \ar[l]^-{-\varphi^*}_-{\simeq} {H^0(\Dual{J}_K,\Omega^1_{\Dual{J}_K/K})} \ar[r]_-{\psi_{J_K}^0}^-{\simeq}  &{H^1(J_K,\O_{J_K})}
					  & \ar[l]^-{H^1(j)^{\dual}} {H^1(X_K,\O_{X_K})^{\dual}}\\
					  & \save[]+<2.3cm,0cm>*{H^0(X_K,\Omega^1_{X_K})} \ar@/_1pc/[rru]_-{(\ref{duality1})}^-{\simeq} \ar@{<-}@/^1pc/[ul]^-{j^*}\restore
				}\label{bigdualdiag}
	  \end{equation}
	  Here, $\psi_{J_K}^0$ is the usual duality (defined using the K\"unneth formula and the first Chern class of the Poincar\'e bundle; see 5.1.3 and Lemme 5.1.4 of \cite{BBM}),
	  the map $\text{ev}$ is the canonical evaluation pairing, and $\text{can}$ is the canonical map obtained by extending sections of $\omega_{J_K}$
	  to invariant differential forms on $J_K$ \cite[\S4.2 Propositions 1, 2]{BLR}.  We claim that each of the three small squares in (\ref{bigdualdiag}) commute, and that the bottom
	  ``sector" anti-commutes.  For the first square, this follows from the fact that the composite $ \Pic^0(j)\circ (-\varphi):J \rightarrow J$ is the identity 
	  map \cite[Lemma  6.9]{MilneJacobians}, together with the fact that the canonical map $\text{can}$ is functorial.  The same reasoning shows that the third 
	  square commutes, using
	  the functoriality of the identification $\theta_K$ \cite[Proposition 1.3 (c)]{LiuNeron}.  The commutativity of the middle square follows immediately from 5.1.1 and the proof of
	  Th\'eor\`eme 5.1.6 in \cite{BBM}.  That the bottom sector region anti-commutes is the content of \cite[Theorem 5.1]{colemanduality}.    
	 Note, as a particular consequence of these commutativity statements, that every map occurring in (\ref{bigdualdiag}) is an isomorphism.	  

	  Using the fact that the canonical duality $\omega_{J_K}\simeq\scrLie(J_K)$ is functorial in $J_K$ (see \cite[Proposition 1.1 (b)]{LiuNeron}) together
	  with the agreement of $\Pic^0(j): \Dual{J}_K\rightarrow J$ and $-\varphi^{-1}$, as above, we conclude that the composite isomorphism 
	  $\Gamma(\Spec K,\omega_{J_K}) \rightarrow \Lie(J_K)^{\vee}$ along the top row of (\ref{bigdualdiag}) is the canonical evaluation duality for $J_K$.  
	  Thus, on generic fibers,
	  the map (\ref{Rthmtrick}) is none other than the map induced by the top, right, and bottom-right edges in (\ref{bigdualdiag}).  But by definition, 
	  the left vertical map of (\ref{can}) coincides with the composite of the left and bottom-left edges of (\ref{bigdualdiag}) on generic fibers, 
	  and is thus the negative of (\ref{Rthmtrick}), as claimed.
	  
	  Now that we know that the left and right vertical maps in (\ref{can}) are isomorphisms when $X$ has rational singularities, it follows that the same is true
	  of all 3 vertical maps, as desired.
\end{proof}

\begin{remark}
	That (\ref{Rthmtrick}) coincides with the left vertical map in (\ref{can}) over generic fibers is essentially Theorem B.4.1 of \cite{GDBC}.  
	We have chosen here to present a different
	  proof because \cite[Theorem B.4.1]{GDBC} rests upon knowing {\em a priori} that the natural pullback map $\Omega^1_{J_K/K}\rightarrow j_*\Omega^1_{X_K/K}$
	  is an isomorphism, while we prefer to deduce this fact as a corollary of our main result.
\end{remark}

\section{Comparison of integral structures}\label{intcomparison}
	
In this section, we use Theorem \ref{main} to prove a comparison result for integral structures in de~Rham cohomology.  
As usual, we fix a discrete valuation ring $R$ with field of fractions $K$.

Let $A_K$ be an abelian variety over $K$.  It is well-known that the Lie algebra of the universal extension $E(\Dual{A}_K)$   
of the dual abelian variety $\Dual{A}_K$ is naturally isomorphic to the first de~Rham cohomology of $A_K$ over $K$, compatibly with
Hodge filtrations:

\begin{proposition}\label{LieUniv}
	There is a canonical isomorphism of short exact sequences of finite dimensional $K$-vector spaces
	\begin{equation*}
		\xymatrix{
			0\ar[r] & {\Lie(\omega_{A_K})} \ar[r]\ar[d]^-{\simeq} & {\Lie(E(\Dual{A}_K))} \ar[r]\ar[d]^-{\simeq} & {\Lie(\Dual{A}_K)} \ar[r]\ar[d]^-{\simeq} & 0 \\
			0\ar[r] & {H^0(A_K,\Omega^1_{A_K/K})} \ar[r] & {H^1_{\dR}(A_K/K)} \ar[r] & {H^1(A,\O_A)} \ar[r] & 0
		}
	\end{equation*}
\end{proposition}

\begin{proof}
	See \cite[I, \S4]{MM}.
\end{proof}

Note that since $\omega_{A_K}$ is a vector group, we have a canonical identification of $\Lie(\omega_{A_K})$ with  the global sections of
$\omega_{A_K}$.
We deduce from Proposition \ref{LieUniv} and Proposition \ref{canonicalext} the following Corollary, which equips the de~Rham cohomology of $A_K$ with a canonical
integral structure:

\begin{corollary}\label{MMintstr}
	Let $A$ $($respectively $\Dual{A}$$)$ be the N\'eron model of $A_K$ $($respectively $\Dual{A}_K$$)$ over $R$
	  and let $\E(\Dual{A})$ be the canonical extension of $\Dual{A}$ $($Definition $\ref{canextdefn})$.
	The sequence of Lie algebras 
	\begin{equation}
		\xymatrix{
			0\ar[r] & {\Lie(\omega_A)} \ar[r] & {\Lie(\E(\Dual{A}))} \ar[r] & {\Lie(\Dual{A})} \ar[r] & 0
		}\label{Lieseqcan}
	\end{equation}
	associated to the canonical extension $(\ref{canextexseq})$ of $A$ over $R$ is a short exact sequence of finite free $R$-modules
	that is contravariantly functorial in $K$-morphisms of abelian varieties $A_K\rightarrow B_K$ over $K$ and recovers
	the $3$-term Hodge filtration of $H^1_{\dR}(A_K/K)$ after extending scalars to $K$.
	That is, $(\ref{Lieseqcan})$ provides a canonical integral structure on the $3$-term Hodge filtration of $H^1_{\dR}(A_K/K)$.
\end{corollary}

\begin{proof}
	Each term in  (\ref{canextexseq}) is a smooth $S$-scheme; in particular
	the map $\scrExtrig_S(A^0,\Gm)\rightarrow \Dual{A}$ is smooth \cite[Exp. $\mathrm{\Rmnum{6}}_{\mathrm{B}}$, Proposition 9.2 vii]{SGA3vol1}.
	Thus, by \cite[Proposition 1.1 (c)]{LiuNeron}, applying the left exact functor $\Lie$ to (\ref{canextexseq}) yields a short exact sequence of 
	finite $R$-modules which are free by smoothness.
	Since any homomorphism of N\'eron models $A\rightarrow B$ induces a map on identity components $A^0\rightarrow B^0$,
	it follows from the N\'eron mapping property and the functoriality of the canonical extension (\ref{canextexseq}) that 
	 (\ref{Lieseqcan}) is contravariantly functorial in $K$-morphisms of abelian varieties $A_K\rightarrow B_K$
	over $K$.  Since the formation of Lie algebras commutes with the scalar extension $R\rightarrow K$,
	we deduce from Proposition \ref{LieUniv} and the fact that $K$-fiber of (\ref{canextexseq}) is the universal extension
	of $\Dual{A}_K$ by a vector group (see Remark \ref{unicaneq}) that (\ref{Lieseqcan}) recovers the Hodge filtration of
	$H^1_{\dR}(A_K/K)$ after extending scalars to $K$.
\end{proof}

\begin{remark}
 	Supposing that Grothendieck's pairing (\ref{Grpairing}) is right non-degenerate,
	so $\scrExtrig_S(A,\Gm)$ is a smooth and separated $S$-scheme by Proposition \ref{repExtrig},
	the natural map of short exact sequences (\ref{idencompmap}) induces an isomorphism on associated exact 
	sequences of Lie algebras by Lemma \ref{idencomp} and \cite[Proposition 1.1 (d)]{LiuNeron}.
\end{remark}

For the remainder of this section, we suppose that $A_K=J_K$ is the Jacobian of a smooth proper and geometrically connected curve $X_K$ over $K$.
Recall that the 3-term Hodge filtration
\begin{equation}
	\xymatrix{
		0\ar[r] & {H^0(X_K,\Omega^1_{X_K/K})} \ar[r] & {H^1_{\dR}(X_K/K)} \ar[r] & {H^1(X_K,\O_{X_K})} \ar[r] & 0
	}
\end{equation}
is auto-dual with respect to the cup product pairing on $H^1_{\dR}(X_K/K)$ and
is contravariantly and covariantly functorial in finite morphisms of smooth and proper curves $g:X_K\rightarrow X_K'$
via pullback $g^*$ and pushforward $g_*$ of differentials. 
Attached to any such morphism, we have associated
homomorphisms of abelian varieties 
\begin{equation*}
	\Pic^0(g) : J_K'\rightarrow J_K \qquad\text{and}\qquad \Alb(g): J_K\rightarrow J_K'
\end{equation*}
by Picard and Albanese functoriality (where $J_K'$ is the Jacobian of $X_K'$).  The following Proposition is well-known, but we have been unable
to find a proof in the literature so we include one here for the convenience of the reader.

\begin{proposition}\label{JacdR}
	There is a canonical isomorphism of short exact sequences of $K$-vector spaces
	\begin{equation}
		\xymatrix{
			0\ar[r] & {H^0(J_K,\Omega^1_{J_K/K})} \ar[r]\ar[d]^-{\simeq} & {H^1_{\dR}(J_K/K)} \ar[r]\ar[d]^-{\simeq} & {H^1(J_K,\O_{J_K})}\ar[r]\ar[d]^-{\simeq} & 0\\
			0\ar[r] & {H^0(X_K,\Omega^1_{X_K/K})} \ar[r] & {H^1_{\dR}(X_K/K)} \ar[r] & {H^1(X_K,\O_{X_K})}\ar[r] & 0
		}\label{Hodgefiliden}
	\end{equation}
	This isomorphism respects the auto-dualities of the top and bottom rows.  Furthermore, 
	 for any finite morphism $g: X_K\rightarrow X_K'$, the map $(\ref{Hodgefiliden})$ intertwines $\Alb(g)^*$ with $g^*$ and $\Pic^0(g)^*$ with $g_*$.
\end{proposition}

\begin{proof}
	We first suppose that $X(K)$ is nonempty and select $x_0\in X(K)$.  Let $j:X_K\rightarrow J_K$ be the associated Albanese morphism.
	By pullback along $j$, we obtain a morphism on de~Rham cohomology
	that yields a commutative diagram (\ref{Hodgefiliden}).  Clearly this map commutes with extension of $K$ (using the same $x_0$) and we claim that
	it is independent of our choice $x_0$.  Each term in the Hodge filtration of $H^1_{\dR}(J_K/K)$ 
	is clearly (the global sections of) a vector group over $K$; denoting any one of them by $V$ it 
	suffices to show that the natural map $J_K\rightarrow \mathbf{Aut}_K(V)$
	given by translations is the zero map.  Since the target is affine of finite type over $K$ and the 
	source is an abelian variety,
	this map factors through a section of the target and must therefore be identically zero, as claimed.
	It follows from Galois descent that we have a canonical map (\ref{Hodgefiliden}) even when $X(K)$ is empty.

	Let us denote by $H(J_K)$ (respectively $H(X_K)$) the three-term exact sequence of $K$-vector spaces given by the top (respectively bottom)
	row of (\ref{Hodgefiliden}).  
	By passing to an extension of $K$ if need be, we may suppose that $X_K(K)$ is nonempty and
	that (\ref{Hodgefiliden}) is given by pullback along an albenese morphism $j:X_K\rightarrow J_K$
	associated to some $x_0\in X_K(K)$.
	To show that (\ref{Hodgefiliden}) is an isomorphism, we will exploit the natural auto-dualities on $H(J_K)$
	and $H(X_K)$.  For this to be successful, it is essential to know that these dualities are compatible:
	
	\begin{lemma}\label{coleduality}
		The canonical auto-dualities of the short exact sequences $H(J_K)$ and $H(X_K)$ are compatible via $j^*$.  That is, the 
		following diagram commutes:
		\begin{equation*}
			\xymatrix{
			{H(J_K)} \ar[d]_-{j^*} & \ar[l]_-{\simeq} {H(J_K)^{\dual}} \\
			{H(X_K)} \ar[r]^-{\simeq} & {H(X_K)^{\dual}} \ar[u]_-{(j^*)^{\vee}}
			}
		\end{equation*}
	\end{lemma}
	
	\begin{proof}
		This is Theorem 5.1 of \cite{colemanduality}.
	\end{proof}	
	
	Continuing with the proof of Proposition \ref{JacdR}, observe that
	 the functoriality of the canonical identification $H^1(X_K,\O_{X_K}) \simeq \Lie(\Pic^0_{X_K/K})$
	yields a commutative diagram
	\begin{equation*}
		\xymatrix{
			{H^1(J_K,\O_{J_K})} \ar[r]^-{\simeq} \ar[d]_-{j^*} & {\Lie(\Pic^0_{J_K/K})} \ar[d]^-{\Lie(\Pic^0(j))} \\
			{H^1(X_K,\O_{X_K})} \ar[r]^-{\simeq} & {\Lie(\Pic^0_{X_K/K})}
		}
	\end{equation*}
	(see \cite[Proposition 1.3 (c)]{LiuNeron}).  Due to \cite[Lemma 6.9]{MilneJacobians}, the map $\Pic^0(j):\Dual{J}_K\rightarrow J_K$
	is the negative of the inverse of the canonical principal polarization $J_K\rightarrow \Dual{J}_K$, so in particular
	it is an isomorphism.  Thus, the map $j^*: H^1(J_K,\O_{J_K})\rightarrow H^1(X_K,\O_{X_K})$ is an isomorphism.
	Taking $K$-linear duals and using the auto-duality of $H(J_K)$ and $H(X_K)$, it follows from Lemma \ref{coleduality}
	that the map $j^*:H^0(J_K,\omega_{J_K})\rightarrow H^0(X_K,\Omega^1_{X_K/K})$ is also an isomorphism.
	We conclude that all three vertical maps of (\ref{Hodgefiliden}) are isomorphisms, as desired.
	
	It remains to check our claims concerning the functoriality of (\ref{Hodgefiliden}) in finite morphisms of smooth proper and geometrically
	connected curves $g:X_K\rightarrow X_K'$.  Denote by $J_K'$ the Jacobian of $X_K'$ and by $j':X_K'\rightarrow J_K'$ the albanese map
	attached to $g(x_0)$.  Albanese functoriality gives a commutative diagram
	\begin{equation}
		\xymatrix{
			X_K\ar[r]^-{g}\ar[d]_-{j} & X'_K\ar[d]^-{j'}\\
			J_K\ar[r]_-{\Alb(g)} & J_K'
		}\label{albb}
	\end{equation}
	from which we easily obtain the commutative diagram of short exact sequences
	\begin{equation}
		\xymatrix{
			H(J_K') \ar[r]^-{\Alb(g)^*}\ar[d]_-{{j'}^*} & H(J_K)\ar[d]^-{j^*}\\
			H(X'_K)\ar[r]_-{g^*} & H(X_K)
		}\label{Albcompat}
	\end{equation}
	which shows that (\ref{Hodgefiliden}) intertwines the maps $g^*$ and $\Alb(g)^*$.  Dualizing (\ref{Albcompat}) and using
	Lemma \ref{coleduality} gives a commutative diagram
		\begin{equation}
		\xymatrix{
			H(J_K)\ar[d]_-{j^*} & \ar[l]_-{\simeq} H(J_K)^{\vee}\ar[r]^-{(\Alb(g)^*)^{\vee}} & H(J'_K)^{\vee}\ar[r]^-{\simeq} &  H(J'_K)\ar[d]^-{{j'}^*} \\
			H(X_K) \ar[r]^-{\simeq} & H(X_K)^{\vee}\ar[u]_-{(j^*)^{\vee}}\ar[r]_-{(g^*)^{\vee}} & H(X'_K)^{\vee}\ar[u]_-{({j'}^*)^{\vee}} & \ar[l]_-{\simeq} H(X'_K)
		}\label{pictrace}
	\end{equation}
	  By \cite[Theorem 5.11 (3)]{CaisDualizing}, the maps $g^*$ and $g_*$ are adjoint with respect to the cup-product pairing on $H^1_{\dR}(X_K/K)$,
	  so the composite map on the bottom row of (\ref{pictrace}) coincides with $g_*$.  We claim that $\Alb(g)^*$ and $\Pic^0(g)^*$ are adjoint
	  with respect to the pairing on $H^1_{\dR}(J_K/K)$, so the top row of (\ref{pictrace}) coincides with $\Pic^0(g)^*$.
	 By definition, the duality pairing on $H^1_{\dR}(J_K/K)$ is deduced from the natural perfect pairing
	 \begin{equation}
	 	\xymatrix{
			{H^1_{\dR}(J_K/K) \times H^1_{\dR}(\Dual{J}_K/K)} \ar[r] & K
		}\label{BBMpair}
	 \end{equation}
	(defined as in \cite[\S5]{BBM}) by identifying the de~Rham cohomology of $\Dual{J}_K$ with that of $J_K$ via the principal polarization
	$J_K\simeq \Dual{J}_K$.  Now if $u:J'_K\rightarrow J_K$ is any morphism with dual $\Dual{u}:\Dual{J}_K\rightarrow \Dual{J}_K'$, then
	the induced maps $u^*$ and $\Dual{u}^*$ on de~Rham cohomology are adjoint with respect to (\ref{BBMpair}) by \cite[5.1.3.3]{BBM}.
	Applying this to $u=\Pic^0(g)$, our claim that $\Alb(g)^*$ and $\Pic^0(g)^*$ are adjoint then follows from the assertion that the composite map
	\begin{equation*}
		\xymatrix{
			{J_K} \ar[r]_-{\varphi}^{\simeq} & {\Dual{J}_K} \ar[r]^-{\Dual{\Pic^0(g)}} & {\Dual{J}_K'} & \ar[l]_-{\simeq}^-{\varphi'} {J_K'}
		}
	\end{equation*}
	coincides with $\Alb(g)$, where $\varphi$ and $\varphi'$ are the canonical principal polarizations.  
	But this follows by applying the functor $\Pic^0$ to the diagram (\ref{albb}) and using the fact that $\Pic^0(j)$ and $\Pic^0(j')$
	coincide with $-\varphi^{-1}$ and $-{\varphi'}^{-1}$, respectively, thanks to \cite[Lemma 6.9]{MilneJacobians}.
\end{proof}

Fix a proper flat and normal model $f:X\rightarrow S$ of $X_K$ over $S=\Spec R$, and denote by
$\omega_{X/S}^{\bullet}$ the two-term  $\O_S$-linear complex of $\O_X$-modules
$d: \O_X\rightarrow \omega_{X/S}$ furnished by Proposition \ref{diffmap}.
We will say that $X$ is an {\em admissible model} of $X_K$ if $X$ has rational 
singularities and $f$ is cohomologically flat in dimension zero.  

Define $H^1(X/R):=\H^1(X,\omega_{X/S}^{\bullet})$.   When $X$ is cohomologically flat, there is a natural
short exact sequence of finite free $R$-modules
\begin{equation}
	\xymatrix{
		0\ar[r] & {H^0(X,\omega_{X/S})} \ar[r] & {H^1(X/R)} \ar[r] & {H^1(X,\O_X)} \ar[r] & 0
		}\label{canonical}
\end{equation}
whose scalar extension to $K$ is identified with the 3-term Hodge filtration exact sequence $H(X_K)$.
Moreover, (\ref{canonical}) is self-dual with respect to the usual cup-product auto-duality of $H(X_K)$ \cite[Proposition 5.8]{CaisDualizing}.
By \cite[Theorem 5.11]{CaisDualizing}, when $X$ is admissible, the integral structure provided by (\ref{canonical}) is canonical:
this short exact sequence is independent of the choice of admissible model $X$ of $X_K$
and is both contravariantly and covariantly functorial via pullback and trace in finite $K$-morphisms 
$X_K\rightarrow X_K'$ of curves having admissible models over $R$.

Via Corollary \ref{MMintstr} and the identification of Hodge filtrations (\ref{Hodgefiliden}), when $X$ is admissible we thus have two canonical integral structures on 
the de~Rham cohomology of $X_K$, and it is natural to ask how these $R$-lattices compare.

\begin{corollary}\label{integralcompare}
	With the notation and hypotheses of Theorem $\ref{main}$, when $X$ has rational singularities
	there is a canonical isomorphism of short exact sequences of finite free $R$-modules
	\begin{equation}
		\xymatrix{
			0\ar[r] & {\Lie(\omega_J)} \ar[r]\ar[d]^-{\simeq} & {\Lie(\E(\Dual{J}))} \ar[r]\ar[d]^-{\simeq} & {\Lie(\Dual{J})} \ar[r]\ar[d]^-{\simeq} & 0\\
			0\ar[r] & {H^0(X,\omega_{X/S})} \ar[r] & {H^1(X/R)} \ar[r] & {H^1(X,\O_X)} \ar[r] & 0
		}\label{intcomp}
	\end{equation}
	that recovers the identification $(\ref{Hodgefiliden})$ after extending scalars to $K$.
\end{corollary}

\begin{remark}
	Let $g:X_K\rightarrow X_K'$ be any finite map of smooth and geometrically connected curves over $K$ and suppose 
	that $X_K$ and $X_K'$ admit proper flat and normal models over $R$ which have rational singularities and generically smooth closed fibers.
	  (such models are automatically admissible due to Raynaud's ``crit\`ere de platitude cohomologique" \cite[Th\'eor\`eme 7.2.1]{RaynaudPic}).
	By our discussion above, $g$ induces maps $g^*$ and $g_*$ on the canonical integral structure (\ref{canonical}) via pullback and trace, and induces
	maps $\Alb(g)^*$ and $\Pic^0(g)^*$ on the canonical integral structure (\ref{Lieseqcan}) by Albanese and Picard functoriality via the N\'eron mapping property.
	The $R$-isomorphism (\ref{intcomp}) necessarily intertwines $\Alb(g)^*$ with $g^*$ and $\Pic^0(g)^*$ with $g_*$ as such agreement of maps between 
	free $R$-modules may be checked after the flat scalar extension $R\rightarrow K$,  where it follows from Proposition \ref{JacdR}.
\end{remark}

\begin{question}
	As an interesting consequence of Corollary \ref{integralcompare}, the duality statement of Proposition \ref{JacdR}, and the fact that the integral
	structure (\ref{canonical}) is auto-dual with respect to cup-product pairing, we deduce that the 
	autoduality of the Hodge filtration of $H^1_{\dR}(J_K/K)$ preserves the integral structure (\ref{Lieseqcan}).  It seems natural to ask if this is true in 
	greater generality, i.e. if for any abelian variety $A_K$ over $K$, the natural duality isomorphism
	\begin{equation*}
		\xymatrix{
			0\ar[r] & {H^1(\Dual{A}_K,\O_{\Dual{A}_K})^{\dual}} \ar[r]\ar[d]^-{\simeq} & {H^1_{\dR}(\Dual{A}_K/K)^{\dual}} \ar[r]\ar[d]^-{\simeq}
			 & {H^0(\Dual{A}_K,\Omega^1_{\Dual{A}_K})} \ar[r]\ar[d]^-{\simeq} & 0\\
			0\ar[r] & {H^0(A_K,\Omega^1_{A_K/K})} \ar[r] & {H^1_{\dR}(A_K/K)} \ar[r] & {H^1(A,\O_A)} \ar[r] & 0
		}
	\end{equation*}
	(see Lemme 5.1.4 and Th\'eor\`eme 5.1.6 of \cite{BBM}) identifies the corresponding canonical integral structures provided by (\ref{Lieseqcan}).
	It is also natural to wonder how such an identification might come about at the level of canonical extensions and N\'eron models
	or more precisely if the definition of the duality for the de~Rham cohomology of an abelian scheme (or more generally a 1-motive)
	in terms of its universal extension (see \cite[p. 636]{coleman-biext} for the case of abelian schemes, \cite{Bertapelle} for 1-motives and \cite[10.2.7.2]{DeligneHodge3}
	for abelian varieties over $\mathbf{C}$) can be extended to the case of N\'eron models and their canonical extensions.
\end{question}

	To prove Corollary $\ref{integralcompare}$, we proceed as follows.
	  By Theorem \ref{main}, we have an isomorphism of short exact sequences of smooth groups as in (\ref{maindiag}).
	Applying the functor $\Lie$ and using the fact that for any group functor $G$ over $S$ with representable fibers, the inclusion $G^0\hookrightarrow G$
	of the identity component induces an isomorphism on Lie algebras \cite[Proposition 1.1 (d)]{LiuNeron}, we deduce a canonical isomorphism
	of free finite $R$-modules 
	\begin{equation*}
		\xymatrix{
			0\ar[r] & {\Lie(\omega_J)} \ar[r]\ar[d]^-{\simeq} & {\Lie(\E(\Dual{J}))} \ar[r]\ar[d]^-{\simeq} & {\Lie(\Dual{J})} \ar[r]\ar[d]^-{\simeq} & 0\\
			0\ar[r] & {\Lie(\omega_{X/S})} \ar[r] & {\Lie(\Pic^{\natural,0}_{X/S})}\ar[r] & {\Lie(\Pic_{X/S})} \ar[r] & 0
	}.
	\end{equation*}
	By Definition \ref{natiden} we have the canonical identifications
	\begin{equation*}
		\Lie(\Pic^{\natural,0}_{X/S}) = \Lie(\Pic^{\natural}_{X/S} \times_{\Pic_{X/S} }\Pic^0_{X/S}) = 
		\Lie(\Pic^{\natural}_{X/S}) \times_{\Lie(\Pic_{X/S})} \Lie(\Pic^0_{X/S})=\Lie(\Pic^{\natural}_{X/S}),
	\end{equation*}
	so it suffices to identify the left exact sequence of Lie algebras attached to 
	the exact sequence of fppf abelian sheaves ({\em cf.} (\ref{close}) )
	\begin{equation}
		\xymatrix{
			0\ar[r] & {f_*\omega_{X/S}}\ar[r] & {\Pic^{\natural}_{X/S}} \ar[r] & {\Pic_{X/S}}
		}\label{passtoH}
	\end{equation}
	with the integral structure on $H(X_K)$ provided by (\ref{canonical}).  
	 As in \S\ref{dualizingsec},  let $\omega_{X_T/T}^{\times,\bullet}$ be the two-term complex 
	 $d_T\log :\O_{X_T}^{\times}\rightarrow \omega_{X_T/T}$ defined by $d_T\log(u)=u^{-1} \cdot d_T(u)$, and write 
	 $\R^1f_*\omega_{X/S}^{\times}$ and $R^1f_*\O_X^{\times}$, respectively, for the fppf sheaves associated to the group functors
	on $S$-schemes 
	\begin{equation*}
		T\rightsquigarrow \H^1(X_T,\omega_{X_T/T}^{\times,\bullet})\qquad\text{and}\qquad T\rightsquigarrow H^1(X_T,\O_{X_T}^{\times}).
	\end{equation*}
	  By Lemma \ref{identify}, the exact sequence (\ref{passtoH})
	is naturally isomorphic to the exact sequence of fppf abelian sheaves
	\begin{equation*}
		\xymatrix{
			0\ar[r] & {f_*\omega_{X/S}}\ar[r] & {\R^1f_* \omega_{X/S}^{\times,\bullet}} \ar[r] & R^1f_*\O_X^{\times}
		}
	\end{equation*}	
	  obtained by sheafifying (\ref{lext}).
	Thus, the proof of Corollary \ref{integralcompare} is completed by:
	
	\begin{lemma}
	 There is a natural isomorphism of exact sequences of free $R$-modules
		\begin{equation}
			\xymatrix{
				0\ar[r] & H^0(X,\omega_{X/S}) \ar[r]\ar[d]^-{\simeq} & H^1(X)\ar[d]^-{\simeq} \ar[r] & H^1(X,\O_X)\ar[r]\ar[d]^-{\simeq} & 0\\
				0\ar[r] & \Lie(f_*\omega_{X/S})\ar[r] & \Lie( \R^1f_*\omega^{\times,\bullet}_{X/S}) \ar[r] & \Lie(R^1f_*\O_{X}^{\times})  & 
			}\label{liecalculation}
		\end{equation}	
	\end{lemma}
	
	\begin{proof}

	By construction, the exact sequence (\ref{canonical}) results from the Hodge to de~Rham spectral sequence attached to 
	the evident filtration of $\omega_{X/S}^{\bullet}$.  
	Now the canonical section $\Z\rightarrow \Z[\epsilon]/(\epsilon^2)$ to the quotient map $\epsilon\mapsto 0$
	induces a canonically split exact sequence of filtered two-term (vertical) complexes
	\begin{equation*}
		\xymatrix{
			0\ar[r] & {\O_{X}}\ar[r]^-{h \mapsto 1+\epsilon h}\ar[d]_-{d} & {\O_{X_{S[\epsilon]}}^{\times}}\ar@<1ex>[r]^-{\epsilon\mapsto 0}\ar[d]_-{d\log} &
			 \ar@<1ex>[l]{\O_{X}^{\times}}\ar[r]\ar[d]^-{d\log} & 1\\
			0\ar[r] & \omega_{X/S}\ar[r]_-{\cdot\epsilon} & \omega_{X_{S[\epsilon]}/S[\epsilon]}\ar@<1ex>[r]^-{\epsilon\mapsto 0} & \ar@<1ex>[l]\omega_{X/S}\ar[r] & 0
		}\label{splitomegas}
	\end{equation*}
	so passing to cohomology yields the commutative diagram with exact rows and columns
	\begin{equation}
		\xymatrix{
	&               0\ar[d]        	                  & 			0\ar[d] 					& 0\ar[d] 				 \\
0\ar[r]& {H^0(X,\omega_{X/S})}\ar[r]\ar[d] &{H^0(X_{S[\epsilon]},\omega_{X_{S[\epsilon]}/S[\epsilon]})}\ar[r]\ar[d] &{H^0(X,\omega_{X/S})}\ar[d]\\ 
0\ar[r]& {\H^1 (X,\omega^{\bullet}_{X/S})}\ar[r]\ar[d] &{\H^1(X_{S[\epsilon]},\omega^{\times,\bullet}_{X_{S[\epsilon]}/S[\epsilon]})}\ar[r]\ar[d] 	
& {\H^1 (X, \omega^{\bullet}_{X/S})}\ar[d] \\
0\ar[r]& {H^1(X,\O_X)}\ar[r] & {H^1(X_{S[\epsilon]},\O_{X_{S[\epsilon]}}^{\times})}\ar[r] & {H^1(X,\O_X^{\times})} \\
		}\label{liemapdefinition}
	\end{equation}
	where the zeroes in the left column result from the splitting (i.e. $\H^1 $ is left exact on {\em split} short exact sequences).
	We conclude that we have an isomorphism of exact sequences of abelian 
	groups as in (\ref{liecalculation}).  It remains to show that this is in fact an $R$-linear isomorphism.
	Recall that for any group functor $G$ on $S$-schemes, 
	the multiplication on $\Lie(G)$ by $\O_S(S)$ is induced by the functoriality of $G$ from the map $\O_S(S)\rightarrow \End_{S}(S[\epsilon])$
	sending $s\in \O_S(S)$ to the self-map $u_s$ of $S[\epsilon]$ that is induced by $\epsilon\mapsto s\cdot\epsilon$.
	Thus, the fact that the map (\ref{liecalculation}) defined by  (\ref{liemapdefinition}) is a map of $R$-modules 
	amounts to the assertion that for any $s\in \O_S(S)$ the following diagram of filtered complexes with exact rows
	\begin{equation}
		\xymatrix{
			0\ar[r] & {\omega^{\bullet}_{X/S}}\ar[r]\ar[d]^-{\cdot s} & {\omega^{\times,\bullet}_{X_{S[\epsilon]}/S[\epsilon]}}\ar[r]\ar[d]^-{u_s^*} & {\omega^{\times,\bullet}_{X/S}}
			\ar[r]\ar[d]^-{\id} & 0\\
			0\ar[r] & {\omega^{\bullet}_{X/S}}\ar[r]  & {\omega^{\times,\bullet}_{X_{S[\epsilon]}/S[\epsilon]}}\ar[r] & {\omega^{\times,\bullet}_{X/S}}\ar[r] & 0
		}
	\end{equation}
	commutes.  This assertion is easily checked.	
\end{proof}

\bibliography{mybib}
\end{document}